\documentclass[12pt,leqno]{article}
\usepackage{amsfonts}
\usepackage{amsmath, amsthm, amsfonts, amssymb, color}
\usepackage{mathrsfs}
\usepackage{url}
\usepackage{xcolor}
\usepackage{exscale}
\usepackage{relsize}

\newtheorem{thm}{Theorem}[section]

\newtheorem{lem}[thm]{Lemma}
\newtheorem{prp}[thm]{Proposition}

\newtheorem{rem}[thm]{Remark}
\theoremstyle{definition}

\newcommand{\scr}[1]{\mathscr #1}
\definecolor{wco}{rgb}{0.5,0.2,0.3}

\numberwithin{equation}{section} \theoremstyle{remark}

\newcommand{\ua}{\uparrow}

\title{\bf Wasserstein  Convergence  for Conditional Empirical Measures of Subordinated Dirichlet Diffusions on Riemannian Manifolds}
\author{Huaiqian Li\footnote{Email: {\color{blue}huaiqianlee@gmail.com}}
\quad Bingyao Wu\footnote{Email: {\color{blue}bingyaowu@163.com}}
  \vspace{2mm}
\\
{\footnotesize Center for Applied Mathematics, Tianjin University, Tianjin 300072, P. R. China}
}

\date{}

\begin{document}
\allowdisplaybreaks
\def\R{\mathbb R}  \def\ff{\frac} \def\ss{\sqrt} \def\B{\mathbf
B}\def\TO{\mathbb T}
\def\I{\mathbb I_{\pp M}}\def\p<{\preceq}
\def\N{\mathbb N} \def\kk{\kappa} \def\m{{\bf m}}
\def\ee{\varepsilon}\def\ddd{D^*}
\def\dd{\delta} \def\DD{\Delta} \def\vv{\varepsilon} \def\rr{\rho}
\def\<{\langle} \def\>{\rangle} \def\GG{\Gamma} \def\gg{\gamma}
  \def\nn{\nabla} \def\pp{\partial} \def\E{\mathbb E}
\def\d{\text{\rm{d}}} \def\bb{\beta} \def\aa{\alpha} \def\D{\scr D}
  \def\si{\sigma} \def\ess{\text{\rm{ess}}}
\def\beg{\begin} \def\beq{\begin{equation}}  \def\F{\scr F}
\def\Ric{{\rm Ric}} \def\Hess{\text{\rm{Hess}}}
\def\e{\text{\rm{e}}} \def\ua{\underline a} \def\OO{\Omega}  \def\oo{\omega}
 \def\ttt{\tilde}\def\tt{\theta}
\def\cut{\text{\rm{cut}}} \def\P{\mathbb P} \def\ifn{I_n(f^{\bigotimes n})}
\def\C{\scr C}      \def\aaa{\mathbf{r}}     \def\r{r}
\def\gap{\text{\rm{gap}}} \def\prr{\pi_{{\bf m},\varrho}}  \def\r{\mathbf r}
\def\Z{\mathbb Z} \def\vrr{\varrho} \def\ll{\lambda}
\def\L{\mathcal{L} }\def\Tt{\tt} \def\TT{\tt}\def\II{\mathbb I}
\def\i{{\rm in}}\def\Sect{{\rm Sect}}  \def\H{\mathbb H}
\def\M{\scr M}\def\Q{\mathbb Q} \def\texto{\text{o}} \def\LL{\Lambda}
\def\Rank{{\rm Rank}} \def\B{\scr B} \def\i{{\rm i}} \def\HR{\hat{\R}^d}
\def\to{\rightarrow}\def\l{\ell}\def\iint{\int}
\def\EE{\scr E}\def\Cut{{\rm Cut}}\def\W{\mathbb W}
\def\A{\scr A} \def\Lip{{\rm Lip}}\def\S{\mathbb S}
\def\BB{\mathbb B}\def\Ent{{\rm Ent}} \def\i{{\rm i}}\def\itparallel{{\it\parallel}}
\def\g{{\mathbf g}}\def\Sect{{\mathcal Sec}}\def\T{\mathcal T}\def\V{{\bf V}}
\def\PP{{\bf P}}\def\HL{{\bf L}}\def\Id{{\rm Id}}\def\f{{\bf f}}\def\cut{{\rm cut}}
\def\Ss{\mathbb S}
\def\BL{\scr A}\def\Pp{\mathbb P}\def\Pp{\mathbb P} \def\Ee{\mathbb E}
\def\hp{\hat{\phi}}\def\vv{\varepsilon}\def\ww{\wedge}
\maketitle

\begin{abstract}
The asymptotic behaviour of empirical measures has plenty of studies. However, the research on conditional empirical measures is limited. Being the development of Wang \cite{eW1}, under the quadratic Wasserstein distance, we investigate the rate of convergence of conditional empirical measures associated to subordinated Dirichlet diffusion processes on a connected compact Riemannian manifold with absorbing boundary. We give the sharp rate of convergence for any initial distribution and prove the precise limit for a large class of initial distributions. We follow the basic idea of Wang, but allow ourselves substantial deviations in the proof to overcome difficulties in our non-local setting.
 \end{abstract} \noindent
{\bf MSC 2020:}  primary  60D05, 58J65; secondary 60J60, 60J76\\
{\bf Keywords:}  Conditional empirical measure, subordinated Dirichlet diffusion process, Riemannian manifold, Wasserstein distance, eigenvalue
 \vskip 2cm

\section{Introduction to main results}
Let $M$ be a $d$-dimensional connected  compact Riemannian manifold with  smooth  boundary $\pp M$. Define $\scr{P}$ as the set of all probability measures on $M$. Let $U\in C^2(M)$ be a potential such that $\mu(\d x)=\e^{U(x)} \d x$ belongs to $\scr{P}$, where $\d x$ is the volume measure on $M$. Denote $\mathcal{L}=\DD+\nn U$, where $\Delta$ and $\nabla$  stand for the Laplace--Beltrami operator and the gradient operator on $M$, respectively. Let $(X_t)_{t\geq0}$ be the diffusion process corresponding to $\mathcal{L}$ with hitting time
$$\tau:=\inf\{t\ge 0:X_t\in\pp M\}.$$

From now on, we use $\P^x$ and $\E^x$ to denote the law and the expectation of the corresponding process with initial position $x\in M$, respectively. For every $\nu\in\scr{P}$, we adopt the usual notation $\P^\nu(\cdot)=\int_M \P^x(\cdot)\,\nu(\d x)$, and denote the expectation w.r.t. $\P^\nu$ by $\E^\nu$. As usual, for every measure $\nu$ on $M$, we use $L^p(\nu):=L^p(M,\nu)$ to denote the $L^p$ function space with norm $\|\cdot\|_{L^p(\nu)}$ for every $p\in[1,\infty]$ and let $\nu(f)$ be the shorthand notation for $\int_M f\d\nu$ for every $f\in L^1(\nu)$.

Let $\rr$ be the geodesic distance on $M$. For any $p\in[1,\infty)$, the $p$-Wasserstein (or $p$-Kantorovich) distance $\W_p:\scr{P}\times\scr{P}\rightarrow[0,\infty]$ is defined as
$$\W_p(\mu_1,\mu_2)= \inf_{\pi\in \C(\mu_1,\mu_2)} \bigg(\int_{M\times M} \rr(x,y)^p \pi(\d x,\d y) \bigg)^{\ff 1 p},\quad \mu_1,\mu_2\in \scr P,$$
where $\C(\mu_1,\mu_2)$ is the set of all probability measures on the product space $M\times M$ with respective
marginal distributions $\mu_1$ and $\mu_2$. $\W_2$ will also be called quadratic Wasserstein distance, on which we mainly focus in this paper. See e.g. \cite[Chapter 5]{ChenMF2004} and \cite{Vill2003} for a detailed study on the $p$-Wasserstein distance and its connection to optimal transportation.

Let $\mathbb{N}_0=\mathbb{N}\cup\{0\}$, where $\mathbb{N}=\{1,2,\cdots\}$. Let $\delta_\cdot$ be the Dirac measure and let $\phi_m,\lambda_m$, $m\in\mathbb{N}_0$, be Dirichlet eigenfunctions and Dirichlet eigenvalues of the operator $-\mathcal{L}$ in $L^2(\mu)$ respectively (see Section 2 below for details). Set
$$\mu_0:=\phi_0^2\mu,$$
which clearly belongs to $\scr P$. Recently, F.-Y. Wang considered the family of  conditional empirical measures
$$\mu_t^\nu:=\E^\nu\Big(\frac{1}{t}\int_0^t\delta_{X_s}\,\d s\Big|\tau>t\Big),\quad t>0,$$
and proved that
\begin{equation}\label{wang-1}
\lim_{t\to\infty}\big\{t^2\W_2(\mu_t^{\nu},\mu_0)^2\big\}
=\ff{1}{\{\mu(\phi_0)\nu(\phi_0)\}^2}\sum_{m=1}^\infty\ff{\{\mu(\phi_0)\nu(\phi_m)+\nu(\phi_0)\mu(\phi_m)\}^2}{(\ll_m-\ll_0)^3},
\end{equation}
where $\nu\in\scr P$ such that $\nu(\partial M)<1$; moreover, he also investigated the finiteness of the limit in \eqref{wang-1}. See \cite[Theorem 1.1]{eW1}.

The aim of the present work is to generalize the above result to a non-local situation for a large class of Markov processes subordinated to $(X_t)_{t\geq0}$. For this purpose, we should recall some basics on subordinated processes. A function $B\in C([0,\infty); [0,\infty))\cap C^\infty((0,\infty);[0,\infty))$ is called a Bernstein function if, for each $ k\in\mathbb N$,
 $$ (-1)^{k-1} \ff{\d^k}{\d t^k}B(t)\ge0, \quad t>0.$$
We will use the following class of  Bernstein functions, i.e.,
$$ {\bf B}:= \big\{B: B\text{\ is\ a\ Bernstein\ function\ with } B(0)=0,\, B'(0)>0\big\}.$$
Let $B\in {\bf  B}$, and let $(S_t^B)_{t\geq0}$ be the unique subordinator corresponding to $B$, i.e., an increasing stochastic process with stationary, independent increments, taking values in $[0, \infty)$ and $S_0^B=0$ such that $B$ is the Laplace exponent of $(S_t^B)_{t\geq0}$, i.e.,
\begin{equation}\label{LT} \E \e^{-\ll S_t^B}= \e^{-t B(\ll)},\ \ t,\ll \ge 0.\end{equation}
Let $(X_t^B)_{t\geq0}$ be the Markov process on $M$ generated by $-B(-\mathcal{L})$. It is well known that $(X_t^B)_{t\geq0}$ can be constructed as the time-changed process of $(X_t)_{t\geq0}$ by $(S_t^B)_{t\geq0}$; more precisely,
 $$X_t^B= X_{S^B_t\wedge\tau},\ \ t\ge 0,$$
where $(S^B_t)_{t\ge 0} $ is the subordinator introduced above,  independent of $(X_t)_{t\ge 0}$. We call $(X_t^B)_{t\geq0}$ the Dirichlet diffusion process subordinated to $(X_t)_{t\geq0}$ or $B$-subordinated Dirichlet diffusion process. See \cite{SSV2012,SV2003,Bertoin97} for some basics and further studies on Bernstein functions and subordinated processes.

Let
$$\si_\tau^B:=\inf\{t\ge0:\ S_t^B>\tau\},$$
which can be regarded as the hitting time of the $B$-subordinate Dirichlet diffusion process $(X_t^B)_{t\geq0}$ at the boundary $\partial M$. For every $t>0$ and every $\nu\in\scr{P}$, we define the conditional empirical measure associated to $(X_t^B)_{t\geq0}$ as
$$\mu_t^{B,\nu}=\E^\nu\left(\left.\ff 1 t\int_0^t \dd_{X_s^B}\d s \right|\si_\tau^B>t\right).$$
Clearly, if $X_0^B\in\partial M$, then $\sigma_0^B=0$. In order to avoid the situation that $\Pp^\nu(\si_\tau^B>t)=0$ for some $\nu\in\scr{P}$, we
should consider the conditional empirical measure $\mu_t^{B,\nu}$ with $\nu\in\scr{P}_0$, where
$$\scr{P}_0:=\{\nu\in \scr{P}:\ \nu(\mathring{M})>0\},$$
and $\mathring{M}:=M\backslash\pp M$, the interior of $M$.

Let $\aa\in [0,1]$ and
 $$\mathbf{B}^\aa:=\left\{B\in \mathbf{B}:\  \liminf_{\ll\to\infty}  \ll^{-\aa} B(\ll)>0\right\}.$$
Recall that $\mu_0=\phi_0^2\mu$. One can verify that, for each $B\in\mathbf{B}^\aa$, $\mu_0$ is the unique quasi-ergodic distribution of the $B$-subordinated Dirichlet diffusion process $(X_t^B)_{t\geq0}$, and for every $\nu\in\scr{P}$ supported on $\mathring{M}$,
$$\|\mu_t^{B,\nu}-\mu_0\|_{\textup{var}}\rightarrow0,\quad t\rightarrow\infty,$$
where $\|\cdot\|_{\textup{var}}$ is the total variation norm; see Appendix for a proof. So, it is interesting to study the large time asymptotic behavior of $\mu_t^{B,\nu}$ to $\mu_0$ in the quadratic Wasserstein distance, and more significantly, estimate the rate of convergence of $\W_2(\mu_t^{B,\nu},\mu_0)$ as $t$ tends to infinity.

Now we are ready to present the main results of this paper. Here and in the sequel, we denote the supremum norm by $\|\cdot\|_\infty$.  Note that  the initial distribution is not required to be supported on $\mathring{M}$.
\begin{thm}\label{T1.1}
Let $\aa\in(0,1]$,  $B\in\mathbf{B}^\aa$  and $\nu\in\scr{P}_0$. Set
$$I:=\ff{4}{\{\mu(\phi_0)\nu(\phi_0)\}^2}\sum_{m=1}^\infty
\ff{\{\nu(\phi_0)\mu(\phi_m)+\mu(\phi_0)\nu(\phi_m)\}^2}{(\ll_m-\ll_0)(B(\ll_m)-B(\ll_0))^2}.$$
Then
\begin{equation}\label{1T1.1}
\limsup_{t\to\infty}\{t^2\W_2(\mu_t^{B,\nu},\mu_0)^2\}\le I\in(0,\infty],
\end{equation}
and moreover, $I$ is finite in either of the following cases:
\begin{itemize}
\item[(1)] $d< 2(1+2\aa)$,
\item[(2)] $d\geq 2(1+2\aa)$, and $\nu=h\mu$ with $h\in L^p(\mu)$ for some $p>\ff{2d}{d+2+4\aa}$ or $h\phi_0^{-1}\in L^q(\mu_0)$ for some $q>\ff{2(d+2)}{d+4+4\aa}$.
\end{itemize}
\end{thm}

Some remarks are in order.
\begin{rem}
(a) Compared with the upper bound in \eqref{wang-1}, the rate of convergence in \eqref{1T1.1} is sharp, although there is an extra factor $4$ in $I$  which  comes from inequality \eqref{W2UB1} below. It seems that the original idea to prove the upper bound in \eqref{wang-1} (see \cite[Section 3]{eW1}) is not applicable to the  present non-local setting due to the difficulty in employing \eqref{W2UB} instead of \eqref{W2UB1}. In order to obtain the precise limit as in \eqref{wang-1}, additional assumption is added on the initial distribution; see Theorem \ref{T1.2} below.

(b) It is interesting to notice that the approach to deal with the finiteness of \eqref{wang-1} is based on an approximation procedure and the non-trivial sharp Sobolev inequality (see e.g. \cite[(1.2)]{eW1}), the latter of which is not available in our non-local setting. As a consequence of our approach, in the particular $\alpha=1$ case, compared with \cite[Theorem 1.1]{eW1}, in order to guarantee $I$ to be finite, when $d=6$ we additionally demand that the initial distribution $\nu$ is absolutely continuous w.r.t. $\mu$ such that the Radon--Nikodym derivative $\frac{\d\nu}{\d\mu}\in L^p(\mu)$ for some $p>1$, and when $d\geq 7$ we require higher order of $L^p(\mu)$ integrability of $\frac{\d\nu}{\d\mu}$ with $p>2d/(d+6)\geq14/13$.

(c) In particular, if $\aa\in(0,1)$ and $B:[0,\infty)\rightarrow[0,\infty)$ such that $B(t)=t^\alpha$, which clearly belongs to $\mathbf{B}^\alpha$, then the corresponding $B$-subordinated Dirichlet diffusion process is the well known $2\alpha$-stable process on $M$ killed upon exiting $\mathring{M}$ (see e.g. \cite{SV2008}). In this case, if $\nu=h\mu$ with $h\in L^2(\mu)$, then the result in Theorem \ref{T1.1}(1) is sharp in the following sense: by \eqref{EIG} and the fact that $\|\phi_k\|_{L^2(\mu)}=1$ for every $k\in\mathbb{N}_0$ (see Section 2 below), there exists a constant $c>0$ such that
$$\ff{1}{\{\mu(\phi_0)\nu(\phi_0)\}^2}\sum_{k=1}^\infty\ff{\{\mu(\phi_0)\nu(\phi_k)+\nu(\phi_0)\mu(\phi_k)\}^2}
{(\ll_k-\ll_0)\big(B(\ll_k)-B(\ll_0)\big)^2}\leq\sum_{k=1}^\infty\frac{c}{k^{2(1+2\alpha)/d}},$$
which is finite if and only if $d<2(1+2\aa)$.
\end{rem}

As \eqref{wang-1}, the next main result contains the precise limit for a large class of initial distributions.
\begin{thm}\label{T1.2}
Let $\aa\in(0,1]$ and $B\in\mathbf{B}^\aa$. Then, for any $\nu=h\mu\in\scr{P}_0$ with $h\phi_0^{-1}\in L^p(\mu_0)$ for some $p\in(p_0,\infty]$,
\begin{equation}\label{T1.2-1}
\lim_{t\to\infty}\{t^2\W_2(\mu_{t}^{B,\nu},\mu_0)^2\}=\ff{1}{\{\mu(\phi_0)\nu(\phi_0)\}^2}\sum_{m=1}^\infty
\ff{[\mu(\phi_0)\nu(\phi_m)+\nu(\phi_0)\mu(\phi_m)]^2}{(\ll_m-\ll_0)[B(\ll_m)-B(\ll_0)]^2}\in(0,\infty),
\end{equation}
where $$p_0:=\ff{6(d+2)}{d+2+12\aa}\vee \ff 3 2.$$
\end{thm}
\begin{rem}
It is clear that the limit in \eqref{T1.2-1} belongs to $(0,\infty)$ follows from Theorem \ref{T1.1}. In order to prove the equality in \eqref{T1.2-1}, a regularization procedure is introduced which leads us to apply \eqref{W2UB} successfully; see Section 4 below for more details. However, this approach seems out of work without additional assumptions on the initial distribution $\nu\in\scr{P}_0$.
\end{rem}

Recently, besides \cite{eW1} mentioned above, large time asymptotic behaviours of empirical measures associated to (subordinated) diffusion processes to the reference measure under Wasserstein distances on Riemannian manifolds have been investigated in a series of papers. (\textsf{i}) Let $M$ be a compact Riemannian manifold with $\partial M$ empty or convex. Uniformly in $x\in M$, the precise limit of $t\E^x[\W_2(\mu_t,\mu)^2]$ as $t\rightarrow\infty$ and sharp rates of  convergence on $\E^x[\W_2(\mu_t,\mu)^2]$ for large $t$ are obtained in \cite{eWZ}, where $\mu_t$ is the empirical measure associated to the given (reflecting) diffusion process (when $\partial M\neq\emptyset$) and $\mu$ is the invariant measure. Furthermore, these results are successfully generalized to subordinated diffusion processes by the second named author joint with F.-Y. Wang in a more recent paper \cite{WangWu}. See related studies on empirical measures in the quadratic Wasserstein distance under the conditional expectation in \cite{eW2}, where the rate of convergence turns out to be quite different from \cite{eW1}. (\textsf{ii}) Let $M$ be a noncompact Riemannian manifold with $\partial M$ empty or convex. Sharp rates of convergence on $\E^\mu[\W_2(\mu_t,\mu)^2]$ for large $t$ are obtained in \cite{eW3}, where $\mu_t$ and $\mu$ are similar as the ones in (\textsf{i}). The results are generalized to a large class of subordinated processes more recently in \cite{LiWu} by the same authors of the present paper. (\textsf{iii}) Refer to \cite{eW4} for further studies on this subject in the setting of stochastic (partial) differential equations.  Last but not the least, being a classical research subject with a wild range of applications, the study on asymptotic behaviours of empirical measures associated to i.i.d. random variables to the reference measure under Wasserstein distances, particularly on quantifying the rate of convergence, has received considerable attention over years; see e.g. the papers \cite{WB2019,eAST,BL2019,Led2017,FG2015,BLG2014,DSS2013,AKT1984} and the book \cite{Tal2014} as well as references therein for many deep results.

The remainder of the paper is laid out as follows. In Section 2, we recall some known results and present some useful properties needed for the later sections. The proof of Theorem \ref{T1.1}  and Theorem \ref{T1.2} are presented in Sections 3 and 4 respectively. We should mention that the present work is motivated by \cite{eW1}. However, we introduce new ideas to overcome difficulties appeared in the present non-local setting.

 \section{Preparations}
In this section, we briefly recall some well known facts on the Dirichlet eigenvalue, the Dirichlet eigenfunction, the Dirichlet diffusion semigroup and the Dirichlet heat kernel, which are mainly borrowed from \cite[Section 2]{eW1}; see e.g. \cite{Wang2014,Ouhabaz05,Davies89,Chavel} for more details. Then we deduce some useful properties on the subordinated Dirichlet diffusion semigroup and introduce necessary notations.

It is well known that the spectrum of the operator $-\mathcal{L}$ is discrete, whose eigenvalues $\lambda_k$, $k\in\mathbb{N}_0$, are nonnegative and listed in an ascending order counting multiplicities, and the corresponding eigenfunctions $\phi_k$, $k\in\mathbb{N}_0$, satisfying the Dirichlet boundary condition, form a complete orthonormal system in the function space $L^2(\mu)$. We may assume that $\phi_0>0$ in $\mathring{M}$ since $\phi_0$ does not change the sign in $\mathring{M}$. It is also well known that $\lambda_0>0$,
\begin{equation}\label{EIG}
\alpha_0^{-1}k^{2/d}\leq\lambda_k-\lambda_0\leq\alpha_0 k^{2/d},\quad \|\phi_k\|_\infty\leq \alpha_0\sqrt{k},\quad\quad k\in\mathbb{N},
\end{equation}
for some constant $\alpha_0>1$, and
\begin{equation}\label{PHI}
\|\phi_0^{-1}\|_{L^p(\mu_0)}<\infty,\quad p\in[1,3).
\end{equation}

Let $p_t^D$ and $P_t^D$ be the Dirichlet heat kernel and the Dirichlet diffusion semigroup corresponding to $\mathcal{L}$, respectively. 
It is well known that $p_t^D$ has the following spectral representation, i.e.,
\begin{equation}\label{DHK}
p_t^D(x,y)=\sum_{m=0}^\infty \e^{-\lambda_m t}\phi_m(x)\phi_m(y),\quad t>0,\,x,y\in M.
\end{equation}
Then, we can use \eqref{DHK} to express the Dirichlet diffusion semigroup as
\begin{equation}\begin{split}\label{DSG}
P_t^D f(x)&:=\E^x[f(X_t)1_{\{t<\tau\}}]=\int_M p_t^D(x,y)f(y)\,\mu(\d y)\\
&=\sum_{m=0}^\infty \e^{-\lambda_m t}\mu(\phi_m f)\phi_m(x),\quad t>0,\,x\in M,\,f\in L^2(\mu).
\end{split}\end{equation}
Moreover, there exists a constant $c>0$ such that
\begin{equation}\begin{split}\label{DPQ}
\|P_t^D\|_{L^p(\mu)\to L^q(\mu)}&:=\sup_{\|f\|_{L^p(\mu)}\leq 1}\|P_t^D f\|_{L^q(\mu)}\\
&\leq c\e^{-\lambda_0 t}(1\wedge t)^{-\frac{d(q-p)}{2pq}},\quad t>0,\, 1\leq p\leq q\leq\infty.
\end{split}\end{equation}

Now consider
 $$\mathcal{L}_0:=\mathcal{L}+2\nabla \log\phi_0.$$
Then $\mathcal{L}_0$ is a non-positive self-adjoint operator in $L^2(\mu_0)$, and the associated semigroup, defined by $P_t^0:=\e^{t\mathcal{L}_0}$ in the sense of functional analysis, satisfies
\begin{equation}\label{R}
P_t^0 f=\e^{\lambda_0 t}\phi_0^{-1}P_t^D(f\phi_0),\quad t\geq0,\,f\in L^2(\mu_0).
\end{equation}
Moreover, $\mu_0$ is the invariant measure of $P_t^0$ since $P_t^0$ is conservative (i.e., $P_t^01=1$ for every $t\geq0$) and symmetric w.r.t.  $\mu_0$. By taking $f=\phi_0^{-1}\phi_k$ in \eqref{R} and noting that $P_t^D\phi_k=\e^{-\ll_k t}\phi_k$ for every $k\in\mathbb{N}_0$,
we clearly see that
\begin{equation}\begin{split}\label{EIG0}
&P_t^0(\phi_k\phi_0^{-1})=\e^{-(\lambda_k-\lambda_0)t}\phi_k\phi_0^{-1},\quad k\in\mathbb{N}_0,\,t\geq 0,\\
&\mathcal{L}_0(\phi_k\phi_0^{-1})=-(\lambda_k-\lambda_0)\phi_k\phi_0^{-1},\quad k\in\mathbb{N}_0,
\end{split}\end{equation}
and hence, $\{\phi_0^{-1}\phi_m\}_{m\in\mathbb{N}_0}$ is an eigenbasis of $-\mathcal{L}_0$ in $L^2(\mu_0)$. Thus, by \eqref{DSG} and \eqref{R},
\begin{equation}\label{SG0}
P_t^0 f=\sum_{m=0}^\infty\mu_0(f\phi_m\phi_0^{-1})\e^{-(\lambda_m-\lambda_0)t}\phi_m\phi_0^{-1},\quad f\in L^2(\mu_0),\,t\geq0,
\end{equation}
and the heat kernel of $P_t^0$ w.r.t. $\mu_0$, denoted by $p_t^0$, can be represented by
\begin{equation}\label{HK0}
p_t^0(x,y)=\sum_{m=0}^\infty(\phi_m\phi_0^{-1})(x)(\phi_m\phi_0^{-1})(y)\e^{-(\lambda_m-\lambda_0)t},\quad x,y\in M,\,t>0.
\end{equation}

By the intrinsic ultra-contractivity (introduced first in \cite{DB1991}),  we can find a constant $\alpha_1\geq 1$ such that
\begin{equation}\begin{split}\label{IU0}
\|P_t^0-\mu_0\|_{L^1(\mu_0)\to L^\infty(\mu_0)}&:=\sup_{\|f\|_{L^1(\mu_0)}\leq 1}\|P_t^0 f-\mu_0(f)\|_{L^\infty(\mu_0)}\\
&\leq\frac{\alpha_1 \e^{-(\lambda_1-\lambda_0)t}}{(1\wedge t)^{(d+2)/2}},\quad t>0,
\end{split}\end{equation}
which along with the semigroup property and the contractivity of $P_t^0$ in $L^p(\mu)$ implies that, there exists a constant $\alpha_2\geq 1$ such that
\begin{equation}\begin{split}\label{PI0}
\|P_t^0-\mu_0\|_{L^p(\mu_0)\rightarrow L^p(\mu_0)}&:=\sup_{\|f\|_{L^p(\mu_0)}\leq 1}\|P_t^0 f-\mu_0(f)\|_{L^p(\mu_0)}\\
&\leq \alpha_2 \e^{-(\lambda_1-\lambda_0)t},\quad t\geq 0,\,\infty \geq p\geq 1.
\end{split}\end{equation}
Combining the Riesz--Thorin interpolation theorem (see e.g. \cite[page 3]{Davies89}) with \eqref{IU0} and \eqref{PI0}, we obtain that
\begin{equation}\label{PQ0}
\|P_t^0-\mu_0\|_{L^p(\mu_0)\to L^q(\mu_0)}\leq \alpha_3 \e^{-(\lambda_1-\lambda_0)t}\{1\wedge t\}^{-\frac{(d+2)(q-p)}{2pq}},\quad t>0,\,\infty\geq q\geq p\geq 1,
\end{equation}
for some constant $\alpha_3>0$. Thus, \eqref{PQ0} and \eqref{EIG} lead to that, there exists a constant $\alpha_4>0$ such that
\begin{equation}\label{EIG0UB}
\|\phi_k\phi_0^{-1}\|_\infty\leq \alpha_4 k^{\frac{d+2}{2d}},\quad k\in\mathbb{N}.
\end{equation}
 Moreover, by \cite[Lemma 2.4]{eW1}, we have
\begin{equation}\label{grad-EIG0UB}
\|\nn(\phi_k\phi_0^{-1})\|_\infty\le \alpha_5 k^{\ff {d+4}{2d}},\quad k\in\mathbb{N},
\end{equation}
for some constant $\alpha_5>0$.

We now turn to the non-local situation. Let $B\in\mathbf{B}$, $t>0$, and let $p_t^{D,B}$ and $P_t^{D,B}$ be the subordinated Dirichlet heat kernel and the subordinated Dirichlet
diffusion semigroup associated with the $B$-subordinated Dirichlet diffusion process $(X_t^B)_{t\geq0}$, respectively. By
\eqref{LT}, \eqref{DHK} and \eqref{DSG}, one has that
\begin{equation}\label{SDHK}
p_t^{D,B}(x,y)=\sum_{m=0}^\infty \e^{-B(\lambda_m) t}\phi_m(x)\phi_m(y),\quad x,y\in M,
\end{equation}
and
\begin{equation}\begin{split}\label{SDSG}
P_t^{D,B} f(x)&:=\E^x[f(X_t^B)1_{\{t<\si_\tau^B\}}]=\int_M p_t^{D,B}(x,y)f(y)\,\mu(\d y)\\
&=\sum_{m=0}^\infty \e^{-B(\lambda_m) t}\mu(\phi_m f)\phi_m(x),\quad x\in M,\,f\in L^2(\mu).
\end{split}\end{equation}
By \eqref{R}, we immediately obtain that
\begin{equation*}\label{R1}
P_t^Df=\e^{-\lambda_0 t}\phi_0P_t^0(f\phi_0^{-1}),\quad f\in L^2(\mu_0).
\end{equation*}
Hence, the semigroup $P_t^{D,B}$ can be written as
\begin{equation}\begin{split}\label{SDSG0}
P_t^{D,B}f&=\int_0^\infty P_s^Df\,\Pp(S_t^B\in\d s)\\
&=\int_0^\infty \e^{-\lambda_0 s}\phi_0P_s^0(f\phi_0^{-1})\,\Pp(S_t^B\in\d s),\quad f\in L^2(\mu_0),
\end{split}\end{equation}
where $\Pp(S_t^B\in\cdot)$ is the distribution of the subordinator $S_t^B$.

We also need the next useful facts. Let $\alpha\in(0,1]$ and $B\in\mathbf{B}^\alpha$. It is easy to see that, there exist constants $a,c>0$ and $b\geq0$ such that
\begin{equation}\label{B-lb}
B(r)\geq c (r^\alpha\wedge r)\geq ar^\aa-b,\quad r\ge 0;
\end{equation}
see also \cite[(3.12)]{WangWu}. Moreover, according to \eqref{B-lb} (which particularly implies that $\lim_{r\rightarrow\infty}B(r)=\infty$), we have for every $r_0\geq0$,
\begin{equation}\label{equ-B}
\lim_{r\rightarrow\infty}\frac{B(r-r_0)}{B(r)-B(r_0)}=1.
\end{equation}
Together with \eqref{B-lb} and \eqref{equ-B}, applying \eqref{SDSG} and \eqref{EIG}, we get a constant $C>0$ such that

\begin{equation}\begin{split}\label{SDHT-U}
&|\e^{B(\ll_0) t}\Pp^\nu(t<\si_\tau^B)-\mu(\phi_0)\nu(\phi_0)|=|\e^{B(\ll_0) t}\nu(P_t^{D,B} 1)-\mu(\phi_0)\nu(\phi_0)|\\
&\leq\sum_{m=1}^\infty \e^{-[B(\lambda_m)-B(\lambda_0)]t}|\mu(\phi_m)\nu(\phi_m)|
\leq \e^{-[B(\lambda_1)-B(\lambda_0)]t/2}\sum_{m=1}^\infty \e^{-[B(\lambda_m)-B(\lambda_0)]t/2}\|\phi_m\|_\infty^2\\
&\leq C\e^{-[B(\lambda_1)-B(\lambda_0)]t/2},\quad t\geq1,\,\nu\in\scr{P}_0,
\end{split}\end{equation}
which clearly implies that
\begin{equation}\begin{split}\label{equ-B1}
\lim_{t\to\infty}\{\e^{B(\ll_0) t}\Pp^\nu(t<\si_\tau^B)\}=\mu(\phi_0)\nu(\phi_0),\quad  \nu\in\scr{P}_0.
\end{split}\end{equation}

The following notation is helpful. Let $\nu\in\scr{P}_0$ and $t>0$. Define
$$\eta_t^\nu=\int_M\phi_0(x)p_t^0(x,\cdot)\nu(\d x),$$
which is obviously non-negative. Then, by \eqref{HK0}, we have the spectral representation of $\eta_t^\nu$ as follows:
\begin{equation}\begin{split}\label{PSI}
\eta_t^\nu&=\sum_{m=0}^\infty\nu(\phi_m)\e^{-(\ll_m-\ll_0)t}\phi_m\phi_0^{-1}\\
&=\nu(\phi_0)+\sum_{m=1}^\infty\nu(\phi_m)\e^{-(\ll_m-\ll_0)t}\phi_m\phi_0^{-1}\geq0.
\end{split}\end{equation}

Let $\B_+(M)$ (resp. $\B_b(M)$) be the class of non-negative (resp. bounded)  measurable functions on $M$, and set $\B_1(M):=\{f\in \B_b(M): \|f\|_\infty\leq1\}$. Denote the standard gamma function as $\Gamma(\cdot)$. For any $a,b\in\R\cup\{\infty,-\infty\}$, $a\wedge b:=\min\{a,b\}$ and $a\vee b:=\max\{a,b\}$; in particular, $a\vee0=:a^+$.

Throughout the following Sections 3 and 4, we always assume that $\alpha\in(0,1]$ and $B\in\mathbf{B}^\alpha$ unless explicitly stated otherwise.

\section{Proofs of Theorem \ref{T1.1}}
In this section, we aim to prove Theorem \ref{T1.1}.  One of the key steps to reach this target is based on the following inequality:
\begin{equation}\label{W2UB1}
\W_2(f\mu_0,\mu_0)^2\leq 4\int_M|\nabla (-\mathcal{L}_0)^{-1}(f-1)|^2\,\d\mu_0,\quad f\geq0,\,\mu_0(f)=1;
\end{equation}
see \cite[Theorem 2]{Led2017}.

Let $\nu\in\scr{P}_0$. In order to employ \eqref{W2UB1} to estimate $\W_2(\mu_t^{B,\nu},\mu_0)$, we should first calculate the Radon--Nikodym derivative $\ff{\d \mu_{t}^{B,\nu}}{\d \mu_0}$. The main tools are the Markov property and the spectral representation of the subordinated Dirichlet diffusion semigroup $(P_t^{D,B})_{t\geq0}$.
\begin{lem}\label{density-h}
Let $\nu\in\scr{P}_0$ and $t>0$.
Then
\begin{equation*}
\ff{\d \mu_{t}^{B,\nu}}{\d \mu_0}=\rho_t^{B,\nu}+1,
\end{equation*}
where
$$\rho_t^{B,\nu}:=\tilde{\rho}_t^{B,\nu}+\ff{1}{t\E^\nu[1_{\{t<\si_\tau^B\}}]}\int_0^t \xi_s\,\d s-A_t,$$
and
\begin{equation}\begin{split}\label{RAI}
\tilde{\rho}_t^{B,\nu}&:=\ff{1}{t\E^\nu[1_{\{t<\si_\tau^B\}}]}\sum_{m=1}^\infty\ff{\big[\mu(\phi_0)\nu(\phi_m)+\nu(\phi_0)\mu(\phi_m)\big]\e^{-B(\ll_0)t}}{B(\ll_m)-B(\ll_0)}\phi_m\phi_0^{-1},\\
A_t&:=\ff{1}{t\E^\nu[1_{\{t<\si_\tau^B\}}]}\sum_{m=1}^\infty\ff{\big[\mu(\phi_0)\nu(\phi_m)+\nu(\phi_0)\mu(\phi_m)\big]\e^{-B(\ll_m)t}}{B(\ll_m)-B(\ll_0)}\phi_m\phi_0^{-1},\\
\xi_s&:=\left(\sum_{m=1}^\infty\e^{-B(\ll_m)s}\nu(\phi_m)\phi_m\phi_0^{-1}\right)
\left(\sum_{n=1}^\infty\e^{-B(\ll_n)(t-s)}\mu(\phi_n)\phi_n\phi_0^{-1}\right)\\
&\quad\quad-\sum_{m=1}^\infty\e^{-B(\ll_m)t}\mu(\phi_m)\nu(\phi_m),\quad 0<s\leq t.
\end{split}\end{equation}
\end{lem}
\begin{proof}
Let $t\ge s>0$ and $f\in\scr{B}_+(M)$. By the Markov property,
\begin{equation*}\begin{split}\label{N1}
&\int_M f\,\d\E^\nu[\dd_{X_s^B}1_{\{t<\si_\tau^B\}}]=\E^\nu[f(X_s^B)1_{\{t<\si_\tau^B\}}]\\
&=\E^\nu\big[f(X_s^B)1_{\{s<\si_\tau^B\}}\E^{X_s^B}(1_{\{t-s<\si_\tau^B\}})\big]\\
&=\E^\nu\big[f(X_s^B)1_{\{s<\si_\tau^B\}}(P_{t-s}^{D,B}1)(X_s^B)\big]\\
&=\nu\big(P_s^{D,B}\{fP_{t-s}^{D,B}1\}\big).
\end{split}\end{equation*}
By \eqref{SDSG}, we have
\begin{align*}
&P_s^{D,B}\{f P_{t-s}^{D,B}1\}(x)=\sum_{m=0}^\infty\e^{-B(\ll_m)s}\mu(\phi_m f P_{t-s}^{D,B}1)\phi_m(x), \quad x\in M.
\end{align*}
Recall that $\mu_0=\phi_0^2\mu$. Applying \eqref{SDSG} again, we derive that
\begin{align*}
\mu(\phi_m f P_{t-s}^{D,B}1)
&=\int_M \phi_m(y)f(y)\left(\sum_{n=0}^\infty\e^{-B(\ll_n)(t-s)}\mu(\phi_n)\phi_n(y)\right)\mu(\d y)\\
&=\sum_{n=0}^\infty\e^{-B(\ll_n)(t-s)}\mu(\phi_n)\int_M\phi_m(y)\phi_n(y)f(y)\mu(\d y)\\
&=\sum_{n=0}^\infty\e^{-B(\ll_n)(t-s)}\mu(\phi_n)\int_M (\phi_m\phi_0^{-1})(y)(\phi_n\phi_0^{-1})(y)f(y)\mu_0(\d y).
\end{align*}
Hence
\begin{equation}\begin{split}\label{N1+}
&\int_M f\,\d\E^\nu[\dd_{X_s^B}1_{\{t<\si_\tau^B\}}]\\
&=\sum_{m=0}^\infty\sum_{n=0}^\infty\e^{-B(\ll_m)s}\e^{-B(\ll_n)(t-s)}\mu(\phi_n)\nu(\phi_m)
\int_M(\phi_m\phi_0^{-1})(y)(\phi_n\phi_0^{-1})(y)f(y)\mu_0(\d y).
\end{split}\end{equation}

According to \eqref{N1+}, we deduce that the Radon--Nikodym derivative of $\E^\nu[\dd_{X_s^B}1_{\{t<\si_\tau^B\}}]$ w.r.t. $\mu_0$ can be written as
\begin{equation}\begin{split}\label{N12}
&\ff{\d\E^\nu[\dd_{X_s^B}1_{\{t<\si_\tau^B\}}]}{\d\mu_0}\\
&=\sum_{m=0}^\infty\sum_{n=0}^\infty\e^{-B(\ll_m)s}\e^{-B(\ll_n)(t-s)}\mu(\phi_n)\nu(\phi_m)(\phi_m\phi_0^{-1})(\phi_n\phi_0^{-1})\\
&=\left(\sum_{m=0}^\infty\e^{-B(\ll_m)s}\nu(\phi_m)\phi_m\phi_0^{-1}\right)\left(\sum_{n=0}^\infty\e^{-B(\ll_n)(t-s)}\mu(\phi_n)\phi_n\phi_0^{-1}\right)\\
&=\e^{-B(\ll_0)t}\mu(\phi_0)\nu(\phi_0)+\e^{-B(\ll_0)s}\nu(\phi_0)\sum_{m=1}^\infty\e^{-B(\ll_m)(t-s)}\mu(\phi_m)\phi_m\phi_0^{-1}\\
&\quad+\e^{-B(\ll_0)(t-s)}\mu(\phi_0)\sum_{m=1}^\infty\e^{-B(\ll_m)s}\nu(\phi_m)\phi_m\phi_0^{-1}\\
&\quad+\left(\sum_{m=1}^\infty\e^{-B(\ll_m)s}\nu(\phi_m)\phi_m\phi_0^{-1}\right)
\left(\sum_{n=1}^\infty\e^{-B(\ll_n)(t-s)}\mu(\phi_n)\phi_n\phi_0^{-1}\right).
\end{split}\end{equation}

Applying \eqref{SDSG} again, we get
\begin{equation*}\begin{split}
\E^\nu[1_{\{t<\si_\tau^B\}}]&=\nu(P_t^{D,B}1)=\sum_{m=0}^\infty\e^{-B(\ll_m)t}\mu(\phi_m)\nu(\phi_m)\\
&=\e^{-B(\ll_0)t}\mu(\phi_0)\nu(\phi_0)+\sum_{m=1}^\infty\e^{-B(\ll_m)t}\mu(\phi_m)\nu(\phi_m).
\end{split}\end{equation*}
Combining this with \eqref{N12}, we have
\begin{align*}
&\ff{\d \E^\nu[\dd_{X_s^B}1_{\{t<\si_\tau^B\}}]}{\d \mu_0}-\E^\nu[1_{\{t<\si_\tau^B\}}]\\
&=\e^{-B(\ll_0)s}\nu(\phi_0)\sum_{m=1}^\infty\e^{-B(\ll_m)(t-s)}\mu(\phi_m)\phi_m\phi_0^{-1}\\
&\quad+\e^{-B(\ll_0)(t-s)}\mu(\phi_0)\sum_{m=1}^\infty\e^{-B(\ll_m)s}\nu(\phi_m)\phi_m\phi_0^{-1}\\
&\quad+\left(\sum_{m=1}^\infty\e^{-B(\ll_m)s}\nu(\phi_m)\phi_m\phi_0^{-1}\right)
\left(\sum_{n=1}^\infty\e^{-B(\ll_n)(t-s)}\mu(\phi_n)\phi_n\phi_0^{-1}\right)\\
&\quad-\sum_{m=1}^\infty\e^{-B(\ll_m)t}\mu(\phi_m)\nu(\phi_m).
\end{align*}
Noting that
\begin{equation*}\begin{split}
{\rm I}&:=\int_0^t \e^{-B(\ll_0)s}\nu(\phi_0)\sum_{m=1}^\infty\e^{-B(\ll_m)(t-s)}\mu(\phi_m)\phi_m\phi_0^{-1}\,\d s\\
&=\sum_{m=1}^\infty\ff{\nu(\phi_0)\mu(\phi_m)
(\e^{-B(\ll_0)t}-\e^{-B(\ll_m)t})}{B(\ll_m)-B(\ll_0)}\phi_m\phi_0^{-1},
\end{split}\end{equation*}
and
\begin{equation*}\begin{split}
{\rm II}&:=\int_0^t\e^{-B(\ll_0)(t-s)}\mu(\phi_0)\sum_{m=1}^\infty\e^{-B(\ll_m)s}\nu(\phi_m)\phi_m\phi_0^{-1}\,\d s\\
&=\sum_{m=1}^\infty\ff{\mu(\phi_0)\nu(\phi_m)(\e^{-B(\ll_0)t}-\e^{-B(\ll_m)t})}{B(\ll_m)-B(\ll_0)}\phi_m\phi_0^{-1},
\end{split}\end{equation*}
we have
\begin{equation*}\begin{split}
{\rm I}+{\rm II}&=\sum_{m=1}^\infty\ff{\big[\mu(\phi_0)\nu(\phi_m)+\nu(\phi_0)\mu(\phi_m)\big]\e^{-B(\ll_0)t}}{B(\ll_m)-B(\ll_0)}\phi_m\phi_0^{-1}\\
&\quad-\sum_{m=1}^\infty\ff{\big[\mu(\phi_0)\nu(\phi_m)+\nu(\phi_0)\mu(\phi_m)\big]\e^{-B(\ll_m)t}}{B(\ll_m)-B(\ll_0)}\phi_m\phi_0^{-1}.
\end{split}\end{equation*}

Therefore, we can write
\begin{align*}
\ff{\d \mu_{t}^{B,\nu}}{\d \mu_0}&=\ff 1 {t\E^\nu[1_{\{t<\si_\tau^B\}}]}\int_0^t\ff{\d \E^\nu[\dd_{X_s^B}1_{\{t<\si_\tau^B\}}]}{\d \mu_0}\,\d s\\
&=\ff 1 {t\E^\nu[1_{\{t<\si_\tau^B\}}]}\int_0^t\Big(\ff{\d \E^\nu[\dd_{X_s^B}1_{\{t<\si_\tau^B\}}]}{\d \mu_0}-\E^\nu[1_{\{t<\si_\tau^B\}}]\Big)\,\d s+1\\
&=\ff 1 {t\E^\nu[1_{\{t<\si_\tau^B\}}]}\left({\rm I}+{\rm II}+\int_0^t  \xi_s\,\d s\right)+1\\
&=\tilde{\rho}_t^{B,\nu}-A_t+\ff 1 {t\E^\nu[1_{\{t<\si_\tau^B\}}]}\int_0^t \xi_s\,\d s+1,
\end{align*}
where $\tilde{\rho}_t^{B,\nu}$, $A_t$ and $\xi_s$ are explicitly expressed in \eqref{RAI}.
\end{proof}

Indeed, we have the following useful integral representation of $\xi_s$.
\begin{rem} Let $\nu\in\scr{P}_0$ and $t>0$. Then
\begin{equation}\begin{split}\label{int-rep-xi}
\xi_s&=\int_0^\infty\int_0^\infty\e^{-\ll_0(k+l)}
[\eta_l^\nu-\nu(\phi_0)][P_k^0\phi_0^{-1}-\mu(\phi_0)] \,\P(S_{t-s}^B\in\d k)\P(S_s^B\in\d l)\\
&\quad-\int_0^\infty \e^{-\ll_0 l}\nu\big(\phi_0\{P_l^0\phi_0^{-1}-\mu(\phi_0)\}\big)\,\P(S_t^B\in\d l),\quad 0<s\leq t.
\end{split}\end{equation}
\begin{proof} Indeed, according to \eqref{SG0},
\begin{equation*}\begin{split}
P_t^0(\phi_0^{-1})&=\sum_{m=0}^\infty\mu_0(\phi_m\phi_0^{-2})\e^{-(\ll_m-\ll_0)t}\phi_m\phi_0^{-1}\\
&=\mu(\phi_0)+\sum_{m=1}^\infty\mu(\phi_m)\e^{-(\ll_m-\ll_0)t}\phi_m\phi_0^{-1},
\end{split}\end{equation*}
which together with \eqref{PSI} and \eqref{LT} implies that
\begin{align*}&\int_0^\infty\int_0^\infty\e^{-\ll_0(k+l)}[\eta_l^\nu-\nu(\phi_0)][P_k^0\phi_0^{-1}-\mu(\phi_0)]\,\P(S_{t-s}^B\in\d k)\P(S_s^B\in\d l)\\
&=\left(\sum_{m=1}^\infty\e^{-B(\ll_m)s}\nu(\phi_m)\phi_m\phi_0^{-1}\right)
\left(\sum_{n=1}^\infty\e^{-B(\ll_n)(t-s)}\mu(\phi_n)\phi_n\phi_0^{-1}\right),\quad 0<s\leq t,
\end{align*}
and
$$\int_0^\infty \e^{-\ll_0 l}\nu\big(\phi_0\{P_l^0\phi_0^{-1}-\mu(\phi_0)\}\big)\,\P(S_t^B\in\d l)=\sum_{m=1}^\infty\e^{-B(\ll_m)t}\mu(\phi_m)\nu(\phi_m).$$
Thus, we immediately obtain \eqref{int-rep-xi} from the definition of $\xi_s$ in \eqref{RAI}.
\end{proof}
\end{rem}

In the next proposition, we establish \eqref{1T1.1} for particular initial distributions.
\begin{prp}\label{Th3.4}
Assume that $B\in\mathbf{B}^\aa$ for some $\aa\in(0,1]$. Then, for every  $\nu\in\scr{P}_0$ satisfying that $\nu=h\mu$ and
$\|h\phi_0^{-1}\|_\infty<\infty$,
\begin{equation*}
\limsup_{t\to\infty}\{t^2\W_2(\mu_t^{B,\nu},\mu_0)^2\}\le \ff{4}{[\mu(\phi_0)\nu(\phi_0)]^2}\sum_{m=1}^\infty\ff{[\mu(\phi_0)\nu(\phi_m)+\nu(\phi_0)\mu(\phi_m)]^2}{(\ll_m-\ll_0)[B(\ll_m)-B(\ll_0)]^2}.
\end{equation*}
\end{prp}
\begin{proof}
Since $\ff{\d\mu_t^{B,\nu}}{\d\mu_0}=1+\rho_t^{B,\nu}$, applying \eqref{W2UB1}, we have
\begin{equation*}
\W_2(\mu_t^{B,\nu},\mu_0)^2\le 4\int_M|\nn(-\L_0)^{-1}\rho_t^{B,\nu}|^2\d\mu_0,\quad t>0.
\end{equation*}
According to Lemma \ref{density-h} and the triangle inequality of $\|\cdot\|_{L^2(\mu_0)}$, we deduce that for any $\dd>0$,
\begin{equation}\begin{split}\label{2Th3.4}
\W_2(\mu_t^{B,\nu},\mu_0)^2&\le 4(1+\dd)\int_M|\nn(-\L_0)^{-1}\tilde{\rho}_t^{B,\nu}|^2\,\d\mu_0\\
&\quad+8(1+\dd^{-1})\int_M|\nn(-\L_0)^{-1} A_t|^2\,\d\mu_0\\
&\quad+8(1+\dd^{-1})\int_M\left|\nn(-\L_0)^{-1}\ff{1}{t\E^\nu[1_{\{t<\si_\tau^B\}}]}\int_0^t \xi_s\,\d s\right|^2\,\d\mu_0.
\end{split}\end{equation}

Since $-\mathcal{L}_0(\phi_m\phi_0^{-1})=(\ll_m-\ll_0)\phi_m\phi_0^{-1}$ and $\|\phi_m\phi_0^{-1}\|_{L^2(\mu_0)}=1$ for every $m\in\mathbb{N}$, by the integration-by-parts formula, we have
\begin{align*}
\int_M|\nn(-\mathcal{L}_0)^{-1}(\phi_m\phi_0^{-1})|^2\,\d\mu_0
=\int_M\phi_m\phi_0^{-1}(-\mathcal{L}_0)^{-1}(\phi_m\phi_0^{-1})\,\d\mu_0=\ff 1 {\ll_m-\ll_0},\quad m\in\mathbb{N}.
\end{align*}
Recalling the definition of $\tilde{\rho}_{t}^{B,\nu}$ and $A_t$ in \eqref{RAI}, it is easy to see that, for every $t>0$,
\begin{equation}\label{3Th3.4}
\int_M|\nn(-\L_0)^{-1}\tilde{\rho}_t^{B,\nu}|^2\d\mu_0=\ff{\e^{-2B(\ll_0)t}}{(t\E^\nu[1_{\{t<\si_\tau^B\}}])^2}\sum_{m=1}^\infty
\ff{[\mu(\phi_0)\nu(\phi_m)+\nu(\phi_0)\mu(\phi_m)]^2}{(\ll_m-\ll_0)[B(\ll_m)-B(\ll_0)]^2},
\end{equation}
and
\begin{equation}\label{4Th3.4}
\int_M|\nn(-\L_0)^{-1}A_t|^2\d\mu_0=\ff 1 {(t\E^\nu[1_{\{t<\si_\tau^B\}}])^2}\sum_{m=1}^\infty
\ff{[\mu(\phi_0)\nu(\phi_m)+\nu(\phi_0)\mu(\phi_m)]^2\e^{-2B(\ll_m)t}}{(\ll_m-\ll_0)[B(\ll_m)-B(\ll_0)]^2}.
\end{equation}

By the definition of $\xi_s$ in \eqref{RAI}, since $\{\phi_m\phi_0^{-1}\}_{m\in\mathbb{N}_0}$ is an orthonormal basis in $L^2(\mu_0)$,
we have
\begin{equation}\begin{split}\label{6Th3.4}
\mu_0\big(\xi_s\big)&=\mu_0\Big(\sum_{m=1}^\infty\sum_{n=1}^\infty\e^{-B(\ll_m)s}\e^{-B(\ll_n)(t-s)}
\nu(\phi_m)\mu(\phi_n)\phi_m\phi_0^{-1}\phi_n\phi_0^{-1}\Big)\\
&\quad-\sum_{m=1}^\infty\e^{-B(\ll_m)t}\mu(\phi_m)\nu(\phi_m)\\
&=\sum_{m=1}^\infty\e^{-B(\ll_m)s}\e^{-B(\ll_m)(t-s)}\mu(\phi_m)\nu(\phi_m)-\sum_{m=1}^\infty\e^{-B(\ll_m)t}\mu(\phi_m)\nu(\phi_m)\\
&=0,\quad 0<s\leq t.
\end{split}\end{equation}
By the fact that $(-\L_0)^{-\ff 1 2}=\ff 2 {\sqrt{\pi}}\int_0^\infty P_{s^2}^0\,\d s$ and Minkowski's inequality, we obtain that
\begin{equation}\begin{split}\label{5Th3.4}
&\int_M\left|\nn(-\L_0)^{-1}\ff{1}{t\E^\nu[1_{\{t<\si_\tau^B\}}]}\int_0^t \xi_s\,\d s\right|^2\,\d\mu_0\\
&=\ff 1 {(t\E^\nu[1_{\{t<\si_\tau^B\}}])^2}\int_M\left|\ff 2 {\sqrt{\pi}}\int_0^\infty\int_0^t P_{r^2}^0\xi_s\,\d s\d r\right|^2\,\d\mu_0\\
&\le \ff 4 {\pi (t\E^\nu[1_{\{t<\si_\tau^B\}}])^2}\left(\int_0^\infty\int_0^t\|P_{r^2}^0\xi_s\|_{L^2(\mu_0)}\,\d s\d r\right)^2,\quad t>0.
\end{split}\end{equation}
Then we can apply \eqref{PI0} to get that
\begin{equation}\begin{split}\label{7Th3.4}
\|P_{r^2}^0\xi_s\|_{L^2(\mu_0)}\le \aa_2\e^{-(\ll_1-\ll_0)r^2}\|\xi_s\|_{L^2(\mu_0)},\quad r>0,
\end{split}\end{equation}
where $\aa_2$ is the same constant in \eqref{PI0}. Recalling the integral representation of $\xi_s$ in \eqref{int-rep-xi}, we have
\begin{equation}\begin{split}\label{8Th3.4}
\|\xi_s\|_{L^2(\mu_0)}&\le\int_0^\infty\int_0^\infty\e^{-\ll_0(k+l)}
\|[\eta_l^\nu-\nu(\phi_0)][P_k^0\phi_0^{-1}-\mu(\phi_0)]\|_{L^2(\mu_0)}\,\P(S_{t-s}^B\in\d k)\P(S_s^B\in\d l)\\
&\quad+\int_0^\infty \e^{-\ll_0 l}|\nu\big(\phi_0\{P_l^0\phi_0^{-1}-\mu(\phi_0)\}\big)|\,\P(S_t^B\in\d l),\quad 0<s\leq t.
\end{split}\end{equation}
Since $\nu=h\mu\in\scr{P}_0$, $\|\phi_0^{-1}\|_{L^2(\mu_0)}=1$, $\eta_s^\nu=P_s^0(h\phi_0^{-1})$ for any $s>0$, and
$\|h\phi_0^{-1}\|_{L^2(\mu_0)}\leq\|h\phi_0^{-1}\|_\infty$, \eqref{PI0} yields that for every $k,l>0$,
\begin{equation}\begin{split}\label{9Th3.4}
&\|[\eta_l^\nu-\nu(\phi_0)][P_k^0\phi_0^{-1}-\mu(\phi_0)]\|_{L^2(\mu_0)}\\
&\le\|\eta_l^\nu-\nu(\phi_0)\|_\infty\|P_k^0\phi_0^{-1}-\mu(\phi_0)\|_{L^2(\mu_0)}\\
&=\|P_l^0(h\phi_0^{-1})-\mu_0(h\phi_0^{-1})\|_\infty\|P_k^0\phi_0^{-1}-\mu(\phi_0)\|_{L^2(\mu_0)}\\
&\le \e^{-(\ll_1-\ll_0)l}\|h\phi_0^{-1}\|_\infty\e^{-(\ll_1-\ll_0)k}\|\phi_0^{-1}\|_{L^2(\mu_0)}\\
&=\|h\phi_0^{-1}\|_\infty\e^{-(\ll_1-\ll_0)(k+l)},
\end{split}\end{equation}
and
\begin{equation}\begin{split}\label{10Th3.4}
|\nu(\phi_0\{\mu(\phi_0)-P_l^0(\phi_0^{-1})\})|&\le\|h\phi_0^{-1}\|_{L^2(\mu_0)}\|P_l^0\phi_0^{-1}-\mu(\phi_0)\|_{L^2(\mu_0)}\\
&\le \|h\phi_0^{-1}\|_\infty\e^{-(\ll_1-\ll_0)l}.
\end{split}\end{equation}
Hence, by \eqref{LT}, \eqref{8Th3.4}, \eqref{9Th3.4} and \eqref{10Th3.4}, we have
\begin{equation*}\begin{split}\label{10+Th3.4}
\|\xi_s\|_{L^2(\mu_0)}&\le\|h\phi_0^{-1}\|_\infty\int_0^\infty\int_0^\infty\e^{-\ll_0(k+l)}\e^{-(\ll_1-\ll_0)(k+l)}\,\P(S_{t-s}^B\in\d k)\P(S_s^B\in\d l)\\
&\quad+\|h\phi_0^{-1}\|_\infty\int_0^\infty\e^{-\ll_0 l}\e^{-(\ll_1-\ll_0)l}\,\P(S_t^B\in\d l)\\
&=2\|h\phi_0^{-1}\|_\infty\e^{-B(\ll_1)t},\quad s>0.
\end{split}\end{equation*}
Substituting this into \eqref{7Th3.4}, we have
\begin{equation}\begin{split}\label{10+Th3.4+}
\|P_{r^2}^0\xi_s\|_{L^2(\mu_0)} \le 2\aa_2\|h\phi_0^{-1}\|_\infty \e^{-B(\ll_1)t}\e^{-(\ll_1-\ll_0)r^2} ,\quad r>0.
\end{split}\end{equation}
Thus, according to \eqref{10+Th3.4+}, \eqref{SDHT-U} and \eqref{5Th3.4}, there \textcolor{red}{exist some constants} $c_1,t_0>0$ such that
\begin{equation}\begin{split}\label{xi-bd}
&\int_M\left|\nn(-\L_0)^{-1}\ff{1}{t\E^\nu[1_{\{t<\si_\tau^B\}}]}\int_0^t \xi_s\,\d s\right|^2\d\mu_0\\
&\le c_1\|h\phi_0^{-1}\|_\infty^2 \e^{-2[B(\ll_1)-B(\ll_0)]t},\quad t\geq t_0.
\end{split}\end{equation}

Therefore, by \eqref{2Th3.4}, \eqref{3Th3.4}, \eqref{4Th3.4} and \eqref{xi-bd}, we find constants $c_2,t_0>0$ such that
\begin{equation}\begin{split}\label{11Th3.4}
t^2\W_2(\mu_t^{B,\nu},\mu_0)^2&\le \ff{4(1+\dd)\e^{-2B(\ll_0)t}}{(\E^\nu[1_{\{t<\si_\tau^B\}}])^2}\sum_{m=1}^\infty
\ff{[\mu(\phi_0)\nu(\phi_m)+\nu(\phi_0)\mu(\phi_m)]^2}{(\ll_m-\ll_0)[B(\ll_m)-B(\ll_0)]^2}\\
&\quad+c_2(1+\dd^{-1})\sum_{m=1}^\infty\ff{[\mu(\phi_0)\nu(\phi_m)+\nu(\phi_0)\mu(\phi_m)]^2}{(\ll_m-\ll_0)[B(\ll_m)-B(\ll_0)]^2}
\e^{-2[B(\ll_m)-B(\ll_0)]t}\\
&\quad+c_2(1+\dd^{-1}) \|h\phi_0^{-1}\|_\infty^2 \e^{-2[B(\ll_1)-B(\ll_0)]t},\quad t\geq t_0.
\end{split}\end{equation}
Since $\|h\phi_0^{-1}\|_\infty<\infty$, by \eqref{equ-B} and \eqref{equ-B1},
letting $t\to\infty$ first and then $\dd\to 0$, we arrive at
$$\limsup_{t\to\infty}\{t^2\W_2(\mu_t^{B,\nu},\mu_0)^2\}\le \ff{4}{[\mu(\phi_0)\nu(\phi_0)]^2}\sum_{m=1}^\infty\ff{[\mu(\phi_0)\nu(\phi_m)+\nu(\phi_0)\mu(\phi_m)]^2}{(\ll_m-\ll_0)[B(\ll_m)-B(\ll_0)]^2}.$$
We finish the proof.
\end{proof}

Now we move to prove Theorem \ref{T1.1}.
\begin{proof}[Proof of Theorem \ref{T1.1}.] We divide the proof into two parts.

(1) Let $\nu\in\scr{P}_0$ and let $t\ge s\ge\vv>0$. By the Markov property, the definition of $\eta_t^\nu$, \eqref{SDSG0} and Fubini's theorem,
we have, for any $f\in\scr{B}_b(M)$,
\begin{equation*}\begin{split}\label{4P3.4}
\E^\nu[f(X_s^B)1_{\{t<\si_\tau^B\}}]&=\E^\nu\big(1_{\{\vv<\si_\tau^B\}}\E^{X_\vv^B}[f(X_{s-\vv}^B)1_{\{t-\vv<\si_\tau^B\}}]\big)\\
&=\int_M P_\varepsilon^{D,B}\psi(x)\,\nu(\d x)\\
&=\int_M\int_0^\infty \e^{-\lambda_0 k}\phi_0(x)P_k^0(\psi\phi^{-1}_0)(x)\,\P(S_\varepsilon^B\in\d k)\nu(\d x)\\
&=\int_0^\infty \e^{-\lambda_0 k} \int_M (\eta_k^\nu\phi_0)(y)\psi(y)\,\mu(\d y)\P(S_\varepsilon^B\in\d k)\\
&=\int_0^\infty \e^{-\lambda_0 k} \int_M (\eta_k^\nu\phi_0)(y)\E^y[f(X_{s-\vv}^B)1_{\{t-\vv<\si_\tau^B\}}]\,\mu(\d y)\P(S_\varepsilon^B\in\d k),
\end{split}\end{equation*}
where we set $\psi(\cdot)=\E^\cdot[f(X_{s-\vv}^B)1_{\{t-\vv<\si_\tau^B\}}]$ for convenience. Taking $f=1$ in this equality, we immediately obtain
$$\P^\nu(t<\si_\tau^B)=\int_0^\infty\int_M\e^{-\ll_0k}(\eta_k^\nu\phi_0)(y)\,\P^y(t-\vv<\sigma_\tau^B)\,\mu(\d y)\P(S_{\vv}^B\in\d k).$$
Letting
$$\nu_\vv=\ff{\int_0^\infty\e^{-\ll_0 k}\eta_k^\nu\phi_0\,\P(S_\vv^B\in\d k)\mu}{\int_0^\infty\e^{-\ll_0 k}\mu(\eta_k^\nu\phi_0)\,\P(S_\vv^B\in\d k)}=:h_\vv\mu,$$
by the Markov property again, we have
\begin{equation*}\begin{split}
&\E^\nu[f(X_s^B)|t<\si_\tau^B]=\ff{\E^\nu[f(X_s^B)1_{\{t<\si_\tau^B\}}]}{\P^\nu(t<\si_\tau^B)}\\
&=\ff{\E^{\nu_\vv}[f(X_{s-\vv}^B)1_{\{t-\vv<\si_\tau^B\}}]}{\P^{\nu_\vv}(t-\vv<\si_\tau^B)}=\E^{\nu_\vv}[f(X_{s-\vv}^B)|t-\vv<\si_\tau^B].
\end{split}\end{equation*}
Thus,
\begin{equation}\label{6P3.4}
\hat{\mu}_{r,\vv}^{B,\nu}:=\ff 1{r-\vv}\int_\vv^r\E^\nu(\dd_{X_s^B}|r<\si_\tau^B)\d s=\mu_{r-\vv}^{B,\nu_\vv},\quad r>\vv.
\end{equation}

Noting that for any $k>0$,
\begin{equation*}\begin{split}
\mu(\eta_k^\nu\phi_0)&=\int_M\int_M p_k^0(x,y)\phi_0(x)\phi_0(y)\,\nu(\d x)\mu(\d y)\\
&=\nu(\phi_0 P_k^0 \phi_0^{-1})\ge\nu(\phi_0)\|\phi_0\|_\infty^{-1}\\
&=:\gg\in(0,1],
\end{split}\end{equation*}
we obtain
$$\int_0^\infty\e^{-\ll_0 k}\mu(\eta_k^\nu\phi_0)\,\P(S_\vv^B\in\d k)\ge \gg\e^{-B(\ll_0)\vv},$$
where we applied \eqref{LT}. Then, by \eqref{LT} and \eqref{B-lb},
\begin{equation}\begin{split}\label{7P3.4-}
&\E[(1\wedge S_\vv^B)^{-\ff{d+2}2}]\le 1+\E\left[(S_\vv^B)^{-\ff{d+2}2}\right]\\
&=1+\ff 1 {\GG(\ff{d+2}2)}\int_0^\infty t^{\ff{d+2}2-1}\e^{-\vv B(t)}\d t\\
&\le 1+\ff 1 {\GG(\ff{d+2}2)}\int_0^\infty t^{\ff{d+2}2-1}\e^{\vv (b-at^\aa)}\d t\\
&\le c\left(1+\vv^{-\ff{d+2}{2\aa}}\right),\quad \varepsilon>0,
\end{split}\end{equation}
for some constant $c>0$. Hence, by the intrinsic ultra-contractivity \eqref{IU0}, there exists a constant $c_1>0$ such that, for every $y\in M$,
\begin{equation*}\begin{split}\label{7P3.4}
|(h_\vv\phi_0^{-1})(y)|&=\frac{\int_0^\infty \e^{-\lambda_0 k}\eta_k^\nu(y)\,\P(S_\varepsilon^B\in\d k)}{\int_0^\infty \e^{-\lambda_0 k}\mu(\eta_k^\nu\phi_0)\,\P(S_\varepsilon^B\in\d k)}\\
&\le \gg^{-1}\e^{B(\ll_0)\vv}\int_0^\infty\e^{-\ll_0 k}\int_M p_k^0(x,y)\phi_0(x)\,\nu(\d x)\P(S_\vv^B\in\d k)\\
&\le \gg^{-1}\e^{B(\ll_0)}\|\phi_0\|_\infty\int_0^\infty\e^{-\ll_0 k}\|p_{k}^0\|_{L^\infty(\mu_0\times \mu_0)}\,\P(S_\vv^B\in\d k)\\
&\le c_1\vv^{-\ff{d+2}{2\aa}},\quad \vv\in(0,1),
\end{split}\end{equation*}
and hence,
\begin{equation}\begin{split}\label{7IP3.4}
\|h_\vv\phi_0^{-1}\|_\infty\le c_1\vv^{-\ff{d+2}{2\aa}},\quad \vv\in(0,1).
\end{split}\end{equation}

Thus,  \eqref{11Th3.4} and \eqref{7IP3.4} imply that, there exist constants $c_2>0$ and $t_0\geq1$ such that, for every $\dd>0$ and every $\vv\in(0,1)$,
\begin{equation}\begin{split}\label{8P3.4}
&t^2\W_2(\hat{\mu}_{t,\vv}^{B,\nu},\mu_0)^2=t^2\W_2(\mu_{t-\vv}^{B,\nu_\varepsilon},\mu_0)^2\\
&\le 4I_\varepsilon +c_2(1+\dd^{-1})\sum_{m=1}^\infty
\ff{[\mu(\phi_0)\nu_\vv(\phi_m)+\nu_\vv(\phi_0)\mu(\phi_m)]^2}{(\ll_m-\ll_0)[B(\ll_m)-B(\ll_0)]^2}\e^{-2(B(\ll_m)-B(\ll_0))(t-\vv)}\\
&\quad+c_2(1+\dd^{-1}) \vv^{-\ff{d+2}\aa}\e^{-2[B(\ll_1)-B(\ll_0)](t-\vv)},\quad  t\geq t_0,
\end{split}\end{equation}
where
$$I_\vv:=\ff 1 {\{\mu(\phi_0)\nu_\vv(\phi_0)\}^2}\sum_{m=1}^\infty\ff{\{\nu_\vv(\phi_0)\mu(\phi_m)+\mu(\phi_0)\nu_\vv(\phi_m)\}^2}
{(\ll_m-\ll_0)[B(\ll_m)-B(\ll_0)]^2}.$$

Next, we prove that
\begin{equation*}\lim_{\vv\to 0^+}I_\vv=\tilde{I}:=\ff 1{\{\mu(\phi_0)\nu(\phi_0)\}^2}\sum_{m=1}^\infty\ff{\{\nu(\phi_0)\mu(\phi_m)+\mu(\phi_0)\nu(\phi_m)\}^2}{(\ll_m-\ll_0)[B(\ll_m)-B(\ll_0)]^2}.
\end{equation*}
By Fubini's theorem, \eqref{R}, \eqref{EIG0} and \eqref{PSI}, for every $k>0$, it is direct to verify that
\begin{align*}
&\mu(\eta_k^\nu \phi_0)=\nu(\phi_0 P_k^{0}\phi_0^{-1})=\e^{\ll_0 k}\nu(P_k^D 1),\\
&\mu(\eta_k^\nu \phi_0\phi_m)=\nu\big(\phi_0 P_k^0(\phi_m\phi_0^{-1})\big)=\e^{-(\ll_m-\ll_0)k}\nu(\phi_m),\quad m\in\mathbb{N}_0.
\end{align*}
Let $\vv\in(0,1)$. Then
$$\nu_\vv(\phi_m)=\ff{\int_0^\infty\e^{-\ll_0 k}\mu(\eta_k^\nu \phi_0\phi_m)\,\P(S_\vv^B\in\d k)}{\int_0^\infty\e^{-\ll_0 k}\mu(\eta_k^\nu \phi_0)\,\P(S_\vv^B\in\d k)}=\ff{\e^{-B(\ll_m) \vv}\nu(\phi_m)}{\nu(P_\vv^{D,B}1)},\quad m\in\mathbb{N}_0.$$
It is easy to see that $\lim_{\vv\to 0^+}\nu_\vv(\phi_m)=\nu(\phi_m)$  for each $m\in\mathbb{N}_0$ since $\nu(P_\vv^{D,B}1)=\P^\nu(\varepsilon<\sigma_\tau^B)\rightarrow1$ as $\vv\to 0^+$. By \eqref{SDSG0} and \eqref{LT}, we have
$$\e^{B(\lambda_m)\varepsilon}\nu(P_\varepsilon^{D,B}1)\geq\e^{B(\lambda_0)\varepsilon}\nu\Big(\int_0^\infty \e^{-\lambda_0 s} \phi_0 P_s^0\phi^{-1}_0\,\P(S_\varepsilon^B\in\d s)\Big)\geq\gamma\in(0,1],\quad m\in\mathbb{N}_0,$$
which together with $\nu(P_\varepsilon^{D,B}1)\leq1$ implies that
\begin{equation}\label{9P3.4}
\e^{-B(\ll_m)\vv}|\nu(\phi_m)|\le|\nu_\vv(\phi_m)|\le
\gamma^{-1}|\nu(\phi_m)|,\quad m\in\mathbb{N}_0.
\end{equation}
Since $\{\phi_m\}_{m\in\mathbb{N}_0}$ is an orthonormal sequence in $L^2(\mu)$, by Bessel's inequality, it is clear that
\begin{equation}\label{10P3.4}
\sum_{m=1}^\infty\mu(\phi_m)^2\le1.
\end{equation}

On the one hand, if $\tilde{I}<\infty$, by \eqref{9P3.4} and \eqref{10P3.4}, applying the dominated convergence theorem, we have $\lim_{\vv\to 0^+}I_\vv=\tilde{I}$.
On the other hand, if $I=\infty$, which is equivalent to
$$\sum_{m=1}^\infty\ff{\nu(\phi_m)^2}{(\ll_m-\ll_0)[B(\ll_m)-B(\ll_0)]^2}=\infty.$$
then, by \eqref{9P3.4} and the monotone convergence theorem, we have
\begin{equation*}\begin{split}
&\liminf_{\vv\to 0^+}\sum_{m=1}^\infty\ff{\nu_\vv(\phi_m)^2}{(\ll_m-\ll_0)[B(\ll_m)-B(\ll_0)]^2}\\
&\ge \liminf_{\vv\to 0^+}\sum_{m=1}^\infty\ff{\e^{-2B(\ll_m) \vv}\nu(\phi_m)^2}{(\ll_m-\ll_0)[B(\ll_m)-B(\ll_0)]^2}=\infty.
\end{split}\end{equation*}
Combining this with \eqref{10P3.4} and $\nu_\vv(\phi_0)\to\nu(\phi_0)$ as $\vv\to 0^+$, we have
\begin{align*}
&\liminf_{\vv\to 0^+}I_\vv=\ff 1{\{\mu(\phi_0)\nu(\phi_0)\}^2}\liminf_{\vv\to 0^+} \sum_{m=1}^\infty\ff{\{\nu_\vv(\phi_0)\mu(\phi_m)+\mu(\phi_0)\nu_\vv(\phi_m)\}^2}{(\ll_m-\ll_0)[B(\ll_m)-B(\ll_0)]^2}\\
&\ge \ff 1 {\{\mu(\phi_0)\nu(\phi_0)\}^2}\liminf_{\vv\to 0^+} \sum_{m=1}^\infty\ff{\{\mu(\phi_0)\nu_\vv(\phi_m)\}^2-2\|\phi_0\|_\infty^2\mu(\phi_m)^2}{2(\ll_m-\ll_0)[B(\ll_m)-B(\ll_0)]^2}=\infty.
\end{align*}
Therefore,
\begin{equation}\label{11P3.4}
\lim_{\vv\to 0^+}I_\vv=\tilde{I}.
\end{equation}

Taking $\vv=t^{-2}$ in \eqref{6P3.4}, by \eqref{8P3.4} and \eqref{11P3.4},  we obtain
\begin{equation}\label{11P3.5}
\limsup_{t\to\infty}\{t^2\W_2(\hat{\mu}_{t,t^{-2}}^{B,\nu},\mu_0)^2\}\le I.
\end{equation}
It is easy to verify that
$$\|\hat{\mu}_{t,\vv}^{B,\nu}-\mu_t^{B,\nu}\|_{\textup{var}}\le\int_\vv^t\left(\ff 1 {t-\vv}-\ff 1 t\right)\d s+\frac{1}{t}\int_0^\vv\d s\le 2\vv t^{-1},\quad \varepsilon\in(0,t),$$
where $\|\mu-\nu\|_{\text{var}}=\sup_{f\in\B_1(M)}|\mu(f)-\nu(f)|$ for any $\mu,\nu\in\scr{P}$ by definition. Then, for any $t>1$,
\begin{equation*}\begin{split}
\W_2(\mu_t^{B,\nu},\hat{\mu}_{t,t^{-2}}^{B,\nu})^2&\le D^2\inf_{\pi\in\scr{C}(\mu_t^{B,\nu},\hat{\mu}_{t,t^{-2}}^{B,\nu})}\pi(\{(x,y)\in M\times M: x\neq y\})\\
&=\ff 1 2 D^2\|\hat{\mu}_{t,t^{-2}}^{B,\nu}-\mu_t^{B,\nu}\|_{\textup{var}}\le  D^2 t^{-3}.
\end{split}\end{equation*}
Combining this with the triangle inequality of $\W_2$, we arrive at
\begin{equation*}\begin{split}
t^2\W_2(\mu_t^{B,\nu},\mu_0)^2\leq (1+\delta^{-1})D^2 t^{-1}+(1+\delta)t^2\W_2(\hat{\mu}_{t,t^{-2}}^{B,\nu},\mu_0)^2,\quad\delta>0.
\end{split}\end{equation*}
Letting $t\rightarrow\infty$ first and then $\delta\rightarrow0^+$, by \eqref{11P3.5}, we finally complete the proof of Theorem \ref{T1.1}.

(2) It is obvious that $I\ge 0$. The same argument in \cite[pages 21-22]{{eW1}} shows that $I>0$.

Now we present the proof of $I<\infty$. Recall that
$$\sum_{m=1}^\infty\mu(\phi_m)^2\le 1,$$
and $B(\ll_m-\ll_0)\overset{m}{\sim} B(\ll_m)-B(\ll_0)$. Then it is easy to see that $I<\infty$ is equivalent to
\begin{equation*}
I':=\sum_{m=1}^\infty\ff{\nu(\phi_m)^2}{(\ll_m-\ll_0)(B(\ll_m-\ll_0))^2}<\infty.
\end{equation*}

(a) Let $d<2+4\aa$ for some $\aa\in(0,1]$. Since
$$\eta_k^\nu-\nu(\phi_0)=\sum_{m=1}^\infty\nu(\phi_m)\e^{-(\ll_m-\ll_0)k}\phi_m\phi_0^{-1},\quad k>0,$$
we have
$$(-\mathcal{L}_0)^{-\ff 1 2}\big(\eta_k^\nu-\nu(\phi_0)\big)=\sum_{m=1}^\infty\ff {\nu(\phi_m)}{\sqrt{\ll_m-\ll_0}}\e^{-(\ll_m-\ll_0)k}\phi_m\phi_0^{-1},\quad k>0,$$
and furthermore,
$$\int_0^\infty (-\mathcal{L}_0)^{-\ff 1 2}\big(\eta_k^\nu-\nu(\phi_0)\big)\,\P(S_r^B\in\d k)=\sum_{m=1}^\infty\ff{\nu(\phi_m)}{\sqrt{\ll_m-\ll_0}}\e^{-B(\ll_m-\ll_0)r}\phi_m\phi_0^{-1}.$$
Hence,
$$\int_0^\infty\int_0^\infty(-\mathcal{L}_0)^{-\ff 1 2}\big(\eta_k^\nu-\nu(\phi_0)\big)\,\P(S_r^B\in\d k)\d r=\sum_{m=1}^\infty\ff{\nu(\phi_m)}{(\ll_m-\ll_0)^{\ff 1 2}B(\ll_m-\ll_0)}\phi_m\phi_0^{-1}.$$
Combining this with the fact that $(-\mathcal{L}_0)^{-\ff 1 2}=\frac{2}{\sqrt{\pi}}\int_0^\infty P_{s^2}^0\,\d s$ and
$$(P^0_{s^2+k/2}-\mu_0)\eta_{k/2}^\nu=\sum_{m=1}^\infty \nu(\phi_m)\e^{-(\lambda_m-\lambda_0)k}\e^{-(\lambda_m-\lambda_0)s^2}\phi_m\phi^{-1}_0=P_{s^2}^0\big(\eta_k^\nu-\nu(\phi_0)\big),$$
we have
\begin{equation}\begin{split}\label{a-1}
\sqrt{I'}&=\left\|\int_0^\infty\int_0^\infty(-\mathcal{L}_0)^{-\ff 1 2}\big(\eta_k^\nu-\nu(\phi_0)\big)\,\P(S_r^B\in\d k)\d r\right\|_{L^2(\mu_0)}\\
&=\frac{2}{\sqrt{\pi}}\left\|\int_0^\infty\int_0^\infty\int_0^\infty(P_{s^2+k/2}^0-\mu_0)\eta_{k/2}^\nu\,\d s\P(S_r^B\in\d k)\d r\right\|_{L^2(\mu_0)}\\
&\le \frac{2}{\sqrt{\pi}}\int_0^\infty\d r\int_0^\infty \P(S_r^B\in\d k)\int_0^\infty\|(P_{s^2+k/2}^0-\mu_0)\eta_{k/2}^\nu\|_{L^2(\mu_0)}\,\d s.
\end{split}\end{equation}

It is easy to see that, by the conservativeness, i.e., $P_t^01=1$ for every $t\geq0$, we get $\mu_0(\eta_{k}^\nu)\leq\|\phi_0\|_\infty<\infty$, $k>0$. By
\eqref{HK0}, \eqref{DHK} and the fact that $P_t^D1\leq 1$ for every $t\geq0$, we have
\begin{eqnarray*}
\|\eta_{k}^\nu\phi_0\|_{L^1(\mu)}&=&\int_M\phi_0(y)\,\mu(\d y)\int_Mp_k^0(x,y)\phi_0(x)\,\nu(\d x)\\
&=&\e^{\lambda_0k}\int_M\int_Mp^D_k(x,y)\,\mu(\d y)\nu(\d x)\leq \e^{\lambda_0k},\quad k>0.
\end{eqnarray*}
Combing these with \eqref{R} and \eqref{DPQ}, we derive that
\begin{equation}\begin{split}\label{a-2}
&\|(P_{s^2+k/2}^0-\mu_0)\eta_{k/2}^\nu\|_{L^2(\mu_0)}\leq \mu_0(\eta_{k/2}^\nu) + \|P_{s^2+k/2}^0\eta_{k/2}^\nu\|_{L^2(\mu_0)}\\
&\leq\|\phi_0\|_\infty+ \e^{(s^2+k/2)\lambda_0}\|P_{s^2+k/2}^D(\eta_{k/2}^\nu\phi_0)\|_{L^2(\mu)}\\
&\leq \|\phi_0\|_\infty+ \e^{(s^2+k/2)\lambda_0}\|P_{s^2+k/2}^D \|_{L^1(\mu)\rightarrow L^2(\mu)}\| \eta_{k/2}^\nu\phi_0 \|_{L^1(\mu)}\\
&\leq c_1(s^2+k/2)^{-d/4},\quad s,k\in E,
\end{split}\end{equation}
for some constant $c_1>0$, where $E:=\{(u,v)\in [0,\infty)\times[0,\infty): u^2+v/2<1\}$. Applying \eqref{PQ0}, we obtain
\begin{equation}\begin{split}\label{a-3}
\|(P_{s^2+k/2}^0-\mu_0)\eta_{k/2}^\nu\|_{L^2(\mu_0)}&\leq \|P_{s^2+k/2}^0-\mu_0\|_{L^1(\mu_0)\rightarrow L^2(\mu_0)}\|\eta_{k/2}^\nu\|_{L^1(\mu_0)}\\
&\leq c_2 \e^{-(\lambda_1-\lambda_0)(s^2+k/2)},\quad s,k\in E^c,
\end{split}\end{equation}
for some constant $c_2>0$, where $E^c:=[0,\infty)\times[0,\infty)\setminus E$.

Hence, by \eqref{a-1}--\eqref{a-3} and the elementary estimate that
\begin{eqnarray*}
&&\int_0^\infty(s^2+k)^{-\frac{d}{4}}\e^{-(\lambda_1-\lambda_0)s^2}\,\d s=\frac{1}{2}\int_0^\infty (u+k)^{-\frac{d}{4}}u^{-\frac{1}{2}} \e^{-(\lambda_1-\lambda_0)u}\,\d u\\
&&=\frac{1}{2}\Big(\int_0^k+\int_k^\infty\Big)(u+k)^{-\frac{d}{4}}u^{-\frac{1}{2}} \e^{-(\lambda_1-\lambda_0)u}\,\d u\leq c_3 k^{-\frac{d-2}{4}},\quad k>0,
\end{eqnarray*}
for some constant $c_3>0$, we have
\begin{equation}\begin{split}\label{a-4}
\sqrt{I'}&\leq c_4\int_0^\infty\Big[\int_E  (s^2+ k/2 )^{-\frac{d}{4}}\,\d s\P(S_r^B\in\d k) + \int_{E^c}
                 \e^{-(\lambda_1-\lambda_0)(s^2+k/2)}\,\d s\P(S_r^B\in\d k)\Big]\d r \\
&\leq c_5  \int_0^\infty\d r\int_0^\infty \P(S_r^B\in\d k)\int_0^\infty (s^2+k/2)^{-\ff {d}{4}}\e^{-(\ll_1-\ll_0)(s^2+k/2)}\,\d s\\
&\leq c_6 \int_0^\infty\d r\int_0^\infty k^{-\ff{d-2}{4}}\e^{-(\ll_1-\ll_0)k/2}\,\P(S_r^B\in\d k),
\end{split}\end{equation}
for some constants $c_4,c_5,c_6>0$.

If $d\in\{1,2\}$, then it is clear that, by \eqref{a-4} and \eqref{LT},  there exist constants $a,b>0$ such that
$$\sqrt{I'}\leq    a \int_0^\infty\e^{-B(b)r}\,\d r<\infty.$$
If $d\in (2,2+4\alpha)$, then \eqref{a-4}, the Cauchy--Schwarz inequality, \eqref{LT} and \eqref{7P3.4-} lead to
\begin{equation*}\begin{split}\label{a-5}
\sqrt{I'}&\leq c_6 \int_0^\infty\Big(\E\big[\e^{-(\lambda_1-\lambda_0)S_r^B}\big]\E\big[(S_r^B)^{-(d-2)/2}\big]\Big)^{1/2}\,\d r\\
&\leq c_7 \int_0^\infty r^{-\frac{d-2}{4\alpha}}\e^{-B(\lambda_1-\lambda_0)r/2} \,\d r<\infty,
\end{split}\end{equation*}
for some constant  $c_7>0$. Thus, $I'<\infty$, provided $d<2+4\alpha$.

(b) Let $d\ge 2+4\aa$ for some $\aa\in(0,1]$ and $\nu=h\mu$ with $h\in L^p(\mu)$ for some $p>\ff{2d}{2+d+4\aa}$.
Since $\{\phi_m\phi_0^{-1}\}_{m\in\mathbb{N}_0}$ is an eigenbasis of $L^2(\mu_0)$, we have
$$h\phi_0^{-1}=\sum_{m=0}^{\infty}\mu_0(h\phi_0^{-1}\phi_m\phi_0^{-1})\phi_m\phi_0^{-1}
=\mu(h\phi_0)+\sum_{m=1}^\infty\mu(h\phi_m)\phi_m\phi_0^{-1},$$
and hence,
\begin{equation*}\begin{split}
(-\mathcal{L}_0)^{-\ff{1+2\aa}2}\big(h\phi_0^{-1}-\mu(h\phi_0)\big)
&=\sum_{m=1}^\infty\ff{\mu(h\phi_m)}{(\ll_m-\ll_0)^{\ff{1+2\aa}2}}\phi_m\phi_0^{-1}\\
&=\sum_{m=1}^\infty\ff{\nu(\phi_m)}{(\ll_m-\ll_0)^{\ff{1+2\aa}2}}\phi_m\phi_0^{-1},
\end{split}\end{equation*}
by \eqref{EIG0}. Thus,
\begin{equation}\label{3P4.4}
I'\leq C\|(-\mathcal{L}_0)^{-\ff{1+2\aa}{2}}\big(h\phi_0^{-1}-\mu(h\phi_0)\big)\|_{L^2(\mu_0)}^2,
\end{equation}
for some constant $C>0$. By \eqref{R} and $\nu=h\mu$, we get
$$P_t^0[h\phi_0^{-1}-\mu_0(h\phi_0^{-1})]=\e^{\ll_0 t}\phi_0^{-1}P_t^D[h-\phi_0\mu_0(h\phi_0^{-1})]=\e^{\ll_0 t}\phi_0^{-1}P_t^D[h-\phi_0\nu(\phi_0)].$$
Combing the fact that $(-\mathcal{L}_0)^{-\ff{1+2\aa}2}=c\int_0^\infty P_{t^{2/(1+2\aa)}}^0\d t$ for some constant $c>0$ with \eqref{DPQ}, \eqref{PQ0},
$\|h\phi_0^{-1}\|_{L^1(\mu_0)}<\infty$, and $\|h\|_{L^p(\mu_0)}<\infty$ for some $p\in(\ff{2d}{d+4+2\aa},2)$, we deduce that
\begin{align*}
&\|(-\mathcal{L}_0)^{-\ff{1+2\aa}{2}}\big(h\phi_0^{-1}-\mu(h\phi_0)\big)\|_{L^2(\mu_0)}\\
&\le c\int_0^\infty\|(P_{t^{2/(1+2\aa)}}^0-\mu_0)\{h\phi_0^{-1}\}\|_{L^2(\mu_0)}\,\d t\\
&\le c\int_0^1\e^{\ll_0 t^{2/(1+2\aa)}}\|\phi_0^{-1}P_{t^{2/(1+2\aa)}}^D\{h-\nu(\phi_0)\phi_0\}\|_{L^2(\mu_0)}\,\d t \\
&\quad +c\|h\phi_0^{-1}\|_{L^1(\mu_0)}\int_1^\infty\|P_{t^{2/(1+2\aa)}}^0-\mu_0\|_{L^1(\mu_0)\to L^2(\mu_0)}\,\d t\\
&=c\int_0^1\e^{\ll_0 t^{2/(1+2\aa)}}\|P_{t^{2/(1+2\aa)}}^D\{h-\nu(\phi_0)\phi_0\}\|_{L^2(\mu)}\,\d t\\
&\quad+c\|h\phi_0\|_{L^1(\mu)}\int_1^\infty\|P_{t^{2/(1+2\aa)}}^0-\mu_0\|_{L^1(\mu_0)\to L^2(\mu_0)}\,\d t\\
&\le c \int_0^1\|P_{t^{2/(1+2\aa)}}^D\|_{L^p(\mu)\to L^2(\mu)}\|h-\nu(\phi_0)\phi_0\|_{L^p(\mu)}\,\d t
+c \int_1^\infty\e^{-(\ll_1-\ll_0)t^{2/(1+2\aa)}}\,\d t\\
&\le c \int_0^1 t^{-\ff{d(2-p)}{2p(1+2\aa)}}\,\d t+c<\infty,
\end{align*}
since $p>\ff{2d}{d+2+4\aa}\geq1$ implies $\ff{d(2-p)}{2p(1+2\aa)}<1$, where the positive constant $c$ may vary from line to line.  If $p\geq2$, then the similar method also leads to $\|(-\mathcal{L}_0)^{-\ff{1+2\aa}{2}}\big(h\phi_0^{-1}-\mu(h\phi_0)\big)\|_{L^2(\mu_0)}<\infty$. Combining these with \eqref{3P4.4}, we  prove the desired result.

(c) Let $d\ge 2+4\aa$ for some  $\aa\in(0,1]$ and $\nu=h\mu$ with $h\phi_0^{-1}\in L^q(\mu_0)$ for some $q>\ff{2(d+2)}{d+4+4\aa}$. By \eqref{3P4.4} and \eqref{PQ0}, if $q<2$, then we have
\begin{align*}
\sqrt{I'}&\leq C\|(-\mathcal{L}_0)^{-\ff{1+2\aa}{2}}\big(h\phi_0^{-1}-\mu(h\phi_0)\big)\|_{L^2(\mu_0)}\\
&\le C\int_0^\infty\|(P_{t^{2/(1+2\aa)}}^0-\mu_0)\{h\phi_0^{-1}\}\|_{L^2(\mu_0)}\,\d t\\
&\le C\int_0^\infty\|P_{t^{2/(1+2\aa)}}^0-\mu_0\|_{L^q(\mu_0)\to L^2(\mu_0)}\|h\phi_0^{-1}\|_{L^q(\mu_0)}\,\d t\\
&\le C\int_0^\infty\{1\wedge t\}^{-\ff{(d+2)(2-q)}{2(1+2\aa)q}}\e^{-(\ll_1-\ll_0)t^{2/(1+2\aa)}}\d t<\infty,
\end{align*}
since $q>\ff{2(d+2)}{d+4+4\aa}$ implies $\ff{(d+2)(2-q)}{d+4+4\aa}<1$, where the positive constant $C$ may vary from line to line. If $q\geq2$, then similarly, we also have $I'<\infty$.

Therefore, the  proof is completed.
\end{proof}

\section{Proofs of Theorem \ref{T1.2}}
This section is divided into two parts. In the first part, we present the proof for the upper bound in Theorem \ref{T1.2}, and in the second part, we prove the lower bound in Theorem \ref{T1.2}.

We begin by introducing some frequently used notations. Let $t,\beta>0$ and $\nu\in\scr{P}_0$. Recall the definition of $\rho_t^{B,\nu}$ and $\tilde{\rho}_t^{B,\nu}$ in Lemma \ref{density-h}. Set
\begin{equation*}
\mu_{t,\beta}^{B,\nu}:=(1+\rho_{t,\beta}^{B,\nu})\mu_0,\quad \rho_{t,\beta}^{B,\nu}:=P_{t^{-\beta}}^0\rho_t^{B,\nu}.
\end{equation*}
 Note that $\mu_{t,\beta}^{B,\nu}$  should be regarded as the regularized version of $\mu_{t}^{B,\nu}$. Similarly, set
\begin{equation*}
\tilde{\mu}_{t,\beta}^{B,\nu}:=(1+\tilde{\rho}_{t,\beta}^{B,\nu})\mu_0,\quad\tilde{\rho}_{t,\beta}^{B,\nu}:=P_{t^{-\beta}}^0\tilde{\rho}_t^{B,\nu}.
\end{equation*}

\subsection{Upper bounds}
Since $\mu_t^{B,\nu}$ is a probability measure, it is easy to see that $\mu_{t,\beta}^{B,\nu}\in\scr{P}_0$. From Lemma \ref{Le4.1} below, for every $\beta\in(0,\ff 2{2d-2\aa+1})$, there exists a constant $t_0>0$ such that $\tilde{\mu}_{t,\beta}^{B,\nu}$ is a probability measure for every $t\geq t_0$. Applying the triangle inequality of $\W_2$, we have
\begin{equation}\begin{split}\label{TRI}
\W_2(\mu_t^{B,\nu},\mu_0)^2&\le(1+\delta)\W_2(\tilde{\mu}_{t,\beta}^{B,\nu},\mu_0)^2+(1+\delta^{-1})\W_2(\tilde{\mu}_{t,\beta}^{B,\nu},\mu_t^{B,\nu})^2\\
&\leq (1+\delta)\W_2(\tilde{\mu}_{t,\beta}^{B,\nu},\mu_0)^2+2(1+\delta^{-1})\W_2(\tilde{\mu}_{t,\beta}^{B,\nu},\mu_{t,\beta}^{B,\nu})^2\\
&\quad+2(1+\delta^{-1})\W_2(\mu_{t,\beta}^{B,\nu},\mu_{t}^{B,\nu})^2,\quad t\geq t_0,\,\delta>0.
\end{split}\end{equation}
Clearly, in order to get the upper bound of $\W_2(\mu_t^{B,\nu},\mu_0)$, we need to estimate the three terms in the right hand side of \eqref{TRI}. The term $\W_2(\tilde{\mu}_{t,\beta}^{B,\nu},\mu_0)$ should be regarded as the dominant term, and the others as error terms.

To bound $\W_2(\tilde{\mu}_{t,\beta}^{B,\nu},\mu_0)$ from above,  the crucial tool is the following inequality: for any probability density functions $f_0$ and $f_1$ w.r.t. $\mu_0$,
\begin{equation}\label{W2UB}
\W_2(f_0\mu_0,f_1\mu_0)^2\leq\int_M\frac{|\nabla (-\mathcal{L}_0)^{-1}(f_0-f_1)|^2}{\mathscr{M}(f_0,f_1)}\,\d\mu_0,
\end{equation}
where the function $\mathscr{M}:[0,\infty)\times[0,\infty)\rightarrow[0,\infty)$ is defined as
$$\mathscr{M}(u,v)=\frac{u-v}{\log u-\log v}1_{\{u\wedge v>0\}},\quad u,v\in[0,\infty),$$
and $\mathscr{M}(u,v)$ is known as the logarithmic mean of $u$ and $v$. Refer to \cite[Proposition 2.3]{eAST} for \eqref{W2UB} and see also \cite[Theorem A.1]{eW3} for general $\W_p$ with $p\geq1$. To bound $\W_2(\tilde{\mu}_{t,\beta}^{B,\nu},\mu_{t,\beta}^{B,\nu})$, we use the total variation norm since $M$ is compact; see Lemma \ref{Lemma4.4} below. For $\W_2(\mu_{t,\beta}^{B,\nu},\mu_t^{B,\nu})$, we apply inequality \eqref{wang-inequ} below.

In order to apply \eqref{W2UB}, the following estimate on $\tilde{\rho}_{t,\beta}^{B,\nu}$ is important.
\begin{lem}\label{Le4.1}
For every $\beta>0$, there exist constants $c,t_0>0$ such that
\begin{equation}\label{1Le4.1}
\|\tilde{\rho}_{t,\beta}^{B,\nu}\|_\infty\leq\ff{c}{\nu(\phi_0)}t^{\ff{(2d-2\aa+1)\beta}2-1},\quad t\geq t_0,\, \nu\in\scr{P}_0.
\end{equation}
Moreover, if $\beta\in(0,\ff 2{2d-2\aa+1})$, then $\tilde{\mu}_{t,\beta}^{B,\nu}$ is a probability measure for every $t\geq t_0$ and every $\nu\in\scr{P}_0$.
\end{lem}
\begin{proof}
Combining
\eqref{SDHT-U} with \eqref{EIG}, \eqref{EIG0}, \eqref{EIG0UB}, \eqref{B-lb},  \eqref{equ-B} and \eqref{RAI}, we can find some constants $c_1,c_2,c_3,c_4,c_5,t_0>0$  such that
\begin{equation*}\begin{split}\label{2Le4.1}
\|\tilde{\rho}_{t,\beta}^{B,\nu}\|_\infty&=\ff 1 {t\E^\nu[1_{\{t<\si_\tau^B\}}]}\left\|\sum_{m=1}^\infty\ff{[\mu(\phi_0)\nu(\phi_m)+\nu(\phi_0)\mu(\phi_m)]\e^{-B(\ll_0)t}}
{[B(\ll_m)-B(\ll_0)]\e^{(\ll_m-\ll_0)t^{-\beta}}}\phi_m\phi_0^{-1}\right\|_\infty\\
&\le\ff {c_1}{t\nu(\phi_0)}\sum_{m=1}^\infty\ff{\|\phi_m\|_\infty\|\phi_m\phi_0^{-1}\|_\infty}{B(\ll_m)-B(\ll_0)}\e^{-(\ll_m-\ll_0)t^{-\beta}}\\
&\le\ff{c_2}{t\nu(\phi_0)}\int_0^\infty s^{\ff{d-2\aa+1}d}\e^{-c_3 s^{\ff 2 d}t^{-\beta}}\,\d s\\
&\le\ff {c_4}{t\nu(\phi_0)}t^{\ff{(2d-2\aa+1)\beta}2}\int_0^\infty u^{\ff{d-2\aa+1}d}\e^{-u^{\ff 2 d}}\,\d u\\
&\le\ff {c_5}{\nu(\phi_0)}t^{\ff{(2d-2\aa+1)\beta}2-1},\quad t\ge t_0,
\end{split}\end{equation*}
which proves \eqref{1Le4.1}.

Now let $\beta\in(0,\ff 2{2d-2\aa+1})$. It is clear that \eqref{1Le4.1} implies that there exists a constant $t_1>0$, such that for any $t\ge t_1$, $\|\tilde{\rho}_{t,\beta}^{B,\nu}\|_\infty\le \ff 1 2$. Hence, $1+\tilde{\rho}_{t,\beta}^{B,\nu}\ge\ff 1 2$ and $\mu_0(1+|\tilde{\rho}_{t,\beta}^{B,\nu}|)\le\ff 3 2$ for every $t\ge t_1$. Noting that $\mu_0(\phi_m\phi_0^{-1})=\mu(\phi_m\phi_0)=0$ for every $m\in\mathbb{N}$, we easily see that $\mu_0(1+\tilde{\rho}_{t,\beta}^{B,\nu})=1$ for any $t>0$. Thus, $\tilde{\mu}_{t,\beta}^{B,\nu}$ is a probability measure for every $t\ge t_1$.
\end{proof}
\begin{rem}\label{density-LB}
In the particular case when $B(r)=r$ for every $r\geq0$, the pointwise lower bound of $\tilde{\rho}_{t}^{B,\nu}$ is obtained in \cite[Lemma 3.2]{eW1}. However, it seems that the same method  of proof does not work in the general setting of Lemma \ref{Le4.1}. That is why we introduce the regularization procedure and establish \eqref{1Le4.1} for $\tilde{\rho}_{t,\beta}^{B,\nu}$. Indeed, the pointwise lower bound in \eqref{1Le4.1} is enough for our purpose.
\end{rem}

In the next lemma, we give an upper bound estimate on $\W_2(\tilde{\mu}_{t,\beta}^{B,\nu},\mu_0)^2$.
\begin{lem}\label{Le4.2}
For every $\beta\in(0,\ff 2 {2d-2\aa+1})$, there exist constants $c,t_0>0$ such that
\begin{equation*}\label{1Le4.2}
t^2\W_2(\tilde{\mu}_{t,\beta}^{B,\nu},\mu_0)^2\le\ff{1+ct^{\ff{(2d-2\aa+1)\beta}{2}-1}}{\{\mu(\phi_0)\nu(\phi_0)\}^2}
\sum_{m=1}^\infty\ff{[\mu(\phi_0)\nu(\phi_m)+\nu(\phi_0)\mu(\phi_m)]^2}{(\ll_m-\ll_0)[B(\ll_m)-B(\ll_0)]^2}\e^{-2(\ll_m-\ll_0)t^{-\beta}},
\end{equation*}
for any $t\ge t_0$ and any $\nu\in\scr{P}_0$.
\end{lem}
\begin{proof}
By Lemma \ref{Le4.1}, for every $\beta\in(0,\ff 2 {2d-2\aa+1})$, there exist constants $c,t_0>0$ such that $\tilde{\mu}_{t,\beta}^{B,\nu}\in\scr{P}_0$ for every $t\ge t_0$, and
$$\scr{M}(1+\tilde{\rho}_{t,\beta}^{B,\nu})\ge 1\wedge(1+\tilde{\rho}_{t,\beta}^{B,\nu})\ge \ff 1 {1+ct^{\ff{(2d-2\aa+1)\beta}{2}-1}},\quad t\ge t_0.$$
So, \eqref{W2UB} implies that
\begin{equation}\begin{split}\label{2L4.2}
\W_2(\tilde{\mu}_{t,\beta}^{B,\nu},\mu_0)^2
&\le\int_M\ff{|\nn(-\L_0)^{-1}\tilde{\rho}_{t,\beta}^{B,\nu}|^2}{\scr{M}(1+\tilde{\rho}_{t,\beta}^{B,\nu},1)}\,\d\mu_0\\
&\le
\left(1+ct^{\ff{(2d-2\aa+1)\beta}{2}-1}\right)\mu_0\big(|\nn(-\L_0)^{-1}\tilde{\rho}_{t,\beta}^{B,\nu}|^2\big),\quad t\ge t_0.
\end{split}\end{equation}
Next, \eqref{EIG0}, \eqref{RAI} and the integration-by-parts formula yield that
\begin{equation*}\begin{split}
&t^2\mu_0(|\nn(-\L_0)^{-1}\tilde{\rho}_{t,\beta}^{B,\nu}|^2)\\
&=\ff 1 {(\E^\nu[1_{\{t<\si_\tau^B\}}])^2}\sum_{m=1}^\infty\ff{[\mu(\phi_0)\nu(\phi_m)+\nu(\phi_0)\mu(\phi_m)]^2\e^{-2B(\ll_0)t}}{(\ll_m-\ll_0)[B(\ll_m)-B(\ll_0)]^2}
\e^{-2(\ll_m-\ll_0)t^{-\beta}}.
\end{split}\end{equation*}
Combining this with \eqref{SDHT-U} and \eqref{2L4.2}, we finish the proof.
\end{proof}

In order to use the total variation norm $\|\mu_{t,\beta}^{B,\nu}-\tilde{\mu}_{t,\beta}^{B,\nu}\|_{\rm var}$ to bound $\W_2(\tilde{\mu}_{t,\beta}^{B,\nu},\mu_{t,\beta}^{B,\nu})$, we need the following lemma.
\begin{lem}\label{L3.1}
For every $p\in( 3/2,\infty]$, there exist constants $c,t_0>0$ such that, for every $\nu\in\scr{P}_0$ with $\nu=h\mu$ and every $t\ge t_0$,
\begin{equation*}\label{1L3.1}
\mu_0(|\rho_t^{B,\nu}-\tilde{\rho}_t^{B,\nu}|)\le c\|h\phi_0^{-1}\|_{L^p(\mu_0)}\e^{-[B(\ll_1)-B(\ll_0)]t}.
\end{equation*}
\end{lem}
\begin{proof}
Let $p\in( 3/2,\infty]$ and $p'$ be its conjugate number. Since  $\eta_s^\nu=P_s^0(h\phi_0^{-1})$ for any $s>0$ and any $\nu=h\mu\in\scr{P}_0$, by H\"{o}lder's inequality and \eqref{PI0}, there exist constants $c_1,c_2>0$ such that
\begin{align*}
&\left\|\big(\eta_l^\nu-\nu(\phi_0)\big) P_k^0[\phi_0^{-1}-\mu(\phi_0)]\right\|_{L^1(\mu_0)}\\
&\le\|\eta_l^\nu-\nu(\phi_0)\|_{L^p(\mu_0)}\|P_k^0\phi_0^{-1}-\mu(\phi_0)\|_{L^{p'}(\mu_0)}\\
&=\|P_l^0(h\phi_0^{-1})-\mu_0(h\phi_0^{-1})\|_{L^p(\mu_0)}\|P_k^0\phi_0^{-1}-\mu_0(\phi_0^{-1})\|_{L^{p'}(\mu_0)}\\
&\le c_1\|h\phi_0^{-1}\|_{L^p(\mu_0)}\|\phi_0^{-1}\|_{L^{p'}(\mu_0)}\e^{-(\ll_1-\ll_0)(k+l)}\\
&\le c_2\|h\phi_0^{-1}\|_{L^p(\mu_0)}\e^{-(\ll_1-\ll_0)(k+l)},\quad k,l>0,
\end{align*}
where we used \eqref{PHI} in the last inequality. Similarly, we find a constant $c_3>0$ such that
\begin{equation}\begin{split}\label{NU}
&|\nu(\phi_0\{\mu(\phi_0)-P_l^0\phi_0^{-1}\})|\\
&\le \|h\phi_0^{-1}\|_{L^p(\mu_0)}\|P_l^0\phi_0^{-1}-\mu(\phi_0)\|_{L^{p'}(\mu_0)}\\
&\le c_3 \|h\phi_0^{-1}\|_{L^p(\mu_0)}\e^{-(\ll_1-\ll_0)l},\quad l>0.
\end{split}\end{equation}
Together with \eqref{SDHT-U}, we find some constants $c_4,t_0>0$ such that
\begin{equation}\begin{split}\label{3L3.1}
&\mu_0\left(\left|\ff 1 {t\E^{\nu}[1_{\{t<\si_\tau^B\}}]}\int_0^t \xi_s\,\d s\right|\right)\le\ff1 {t\E^{\nu}[1_{\{t<\si_\tau^B\}}]}\int_0^t\|\xi_s\|_{L^1(\mu_0)}\,\d s\\
&\le \ff {c_4 \e^{B(\ll_0) t}} t\int_0^t\int_0^\infty\int_0^\infty\e^{-\ll_0(k+l)}\left\|[\eta_l^\nu-\nu(\phi_0)]P_k^0[\phi_0^{-1}-\mu(\phi_0)]\right\|_{L^1(\mu_0)}\\
&\quad\P(S_{t-s}^B\in\d k)\P(S_s^B\in\d l)\d s\\
&\quad+\ff {c_4 \e^{B(\ll_0) t}} t\int_0^t\int_0^\infty \e^{-\ll_0 l}\|\nu(\phi_0\{P_l^0[\phi_0^{-1}-\mu(\phi_0)]\})\|_{L^1(\mu_0)}\,\P(S_t^B\in\d l)\d s\\
&\le \ff {c_4 \e^{B(\ll_0) t}} t \|h\phi_0^{-1}\|_{L^p(\mu_0)}
\int_0^t\int_0^\infty\int_0^\infty\e^{-\ll_1(k+l)}\,\P(S_{t-s}^B\in\d k)\P(S_s^B\in\d l)\d s\\
&\quad +\ff {c_4 \e^{B(\ll_0) t}} t \|h\phi_0^{-1}\|_{L^p(\mu_0)}
\int_0^t\int_0^\infty \e^{-\ll_1 l}\,\P(S_t^B\in\d l)\d s\\
&=2c_4\|h\phi_0^{-1}\|_{L^p(\mu_0)}
\e^{-[B(\ll_1)-B(\ll_0)]t},\quad t\geq t_0,
\end{split}\end{equation}
where we used \eqref{int-rep-xi} in the first inequality and \eqref{LT} in the last equality. By \eqref{RAI}, \eqref{SDHT-U}, \eqref{equ-B}, \eqref{B-lb}, \eqref{EIG} and the fact that $\|\phi_m\phi_0^{-1}\|_{L^2(\mu_0)}=1$ for every $m\in\mathbb{N}$, we find some constants $c_5,c_6, t_0>0$
such that
\begin{equation*}\begin{split}\label{4L3.1}
\mu_0(|A_t|)&\le \ff{c_5} t\sum_{m=1}^\infty \ff{\|\phi_m\|_\infty\e^{-[B(\ll_m)-B(\ll_0)]t}}{B(\ll_m)-B(\ll_0)}\mu_0(|\phi_m\phi_0^{-1}|)\\
&\le c_6\e^{-[B(\ll_1)-B(\ll_0)]t},\quad t\ge t_0.
\end{split}\end{equation*}
Combining this with \eqref{3L3.1} and
$$\mu_0(|\rho_t^{B,\nu}-\tilde{\rho}_t^{B,\nu}|)\leq \mu_0\left(\left|\ff 1 {t\E^{\nu}[1_{\{t<\si_\tau^B\}}]}\int_0^t \xi_s\,\d s\right|\right)+\mu_0(|A_t|),$$
we complete the proof of the desired result.
\end{proof}

With Lemma \ref{L3.1} at our disposal, we may obtain upper bound estimates on $\W_2(\tilde{\mu}_{t,\beta}^{B,\nu},\mu_{t,\beta}^{B,\nu})$.
\begin{lem}\label{Lemma4.4}
For every $p\in(3/2,\infty]$ and every $\beta\in(0,\ff 2{2d-2\aa+1})$, there exist constants $t_0,c>0$ such that
$$\W_2(\mu_{t,\beta}^{B,\nu},\tilde{\mu}_{t,\beta}^{B,\nu})^2\le c\|h\phi_0^{-1}\|_{L^p(\mu_0)}\e^{-[B(\ll_1)-B(\ll_0)]t},\quad  t\ge t_0,\,\nu=h\mu\in\scr{P}_0.$$
\end{lem}
\begin{proof}
We use $D$ to denote the diameter of $M$, i.e., $D=\sup{\{\rho(x,y):x,y\in M\}}$, which is obviously finite since $M$ is compact. According to Lemmas \ref{Le4.1} and \ref{L3.1}, for every $p\in(3/2,\infty]$ and every $\beta\in(0,\ff 2{2d-2\aa+1})$, there exist constants $t_0,c>0$ such that,
$\tilde{\mu}_{t,\beta}^{B,\nu}$ is a probability measure for any $t\ge t_0$ and $\nu\in\scr{P}_0$, and
\begin{equation*}\begin{split}
\W_2(\mu_{t,\beta}^{B,\nu},\tilde{\mu}_{t,\beta}^{B,\nu})^2
&\le \ff 1 2 D^2\|\mu_{t,\beta}^{B,\nu}-\tilde{\mu}_{t,\beta}^{B,\nu}\|_{\textup{var}}=\ff 1 2D^2\mu_0(|\rho_{t,\beta}^{B,\nu}-\tilde{\rho}_{t,\beta}^{B,\nu}|)\\
&\le c \|h\phi_0^{-1}\|_{L^p(\mu_0)}\e^{-[B(\ll_1)-B(\ll_0)]t},\quad t\ge t_0,\,\nu=h\mu\in\scr{P}_0,
\end{split}\end{equation*}
where in the first equality we used the fact that
$$\|\mu_{t,\beta}^{B,\nu}-\tilde{\mu}_{t,\beta}^{B,\nu}\|_{\textup{var}}=\Big\|\ff{\d \mu_{t,\beta}^{B,\nu}}{\d \mu_0}-\ff{\d \tilde{\mu}_{t,\beta}^{B,\nu}}{\d \mu_0}\Big\|_{L^1(\mu_0)}=\|\rho_{t,\beta}^{B,\nu}-\tilde{\rho}_{t,\beta}^{B,\nu}\|_{L^1(\mu_0)}.$$
\end{proof}

Next, we estimate the error term $\W_2(\mu_t^{B,\nu},\mu_{t,\beta}^{B,\nu})$. We need the following inequality borrowed from \cite[Theorem A.1]{eW3} (see also \cite[Corollary 4.4]{AG2019}), i.e.,
for any probability density functions $f_1$ and $f_2$ with respect to $\mu_0$ such that $f_1\vee f_2>0$,
\begin{equation}\label{wang-inequ}
\W_2(f_1\mu_0,f_2\mu_0)^2\le 4\int_M\ff{|\nabla(-\L_0)^{-1}(f_2-f_1)|^2}{f_1}\,\d\mu_0.
\end{equation}
Recall the number $p_0$ introduced in Theorem \ref{T1.2}, i.e.,
$$p_0=\ff{6(d+2)}{d+2+12\aa}\vee \ff 3 2.$$
\begin{lem}\label{Le4.6}
Let $\beta\in(0,\ff 2{2d-2\aa+1})$ and $p\in(p_0,\infty]$. Then there exist constants $c,t_0>0$ such that, for every $t\ge t_0$ and every $\nu=h\mu\in\scr{P}_0$ with $h\phi_0^{-1}\in L^p(\mu_0)$,
$$t^2\W_2(\mu_t^{B,\nu},\mu_{t,\beta}^{B,\nu})^2\le c \|h\phi_0^{-1}\|_{L^p(\mu_0)}^2 t^{-\beta}.$$
\end{lem}

\begin{proof} We divide the proof into three parts.

(1) Since $\mu_0(\xi_s)=0$ for any $s>0$ (see \eqref{6Th3.4} above) and $p_0\ge 3/2$, by \eqref{IU0} and \eqref{3L3.1}, we find some constants $c_1,c_2,t_0>0$ such that
\begin{equation}\begin{split}\label{1Le4.6}
&\left\|\ff 1 {t\E^\nu[1_{\{t<\si_\tau^B\}}]}\int_0^t P_{t^{-\beta}}^0\xi_s\,\d s\right\|_\infty\le\ff 1 {t\E^\nu[1_{\{t<\si_\tau^B\}}]}\int_0^t\|P_{t^{-\beta}}^0\xi_s\|_\infty\,\d s\\
&\le \ff 1 {t\E^\nu[1_{\{t<\si_\tau^B\}}]}\int_0^t \|P_{t^{-\beta}}^0-\mu_0\|_{L^1(\mu_0)\to L^\infty(\mu_0)}\|\xi_s\|_{L^1(\mu_0)}\,\d s\\
&\le c_1 \e^{-(\ll_1-\ll_0)t^{-\beta}}\{1\wedge t^{-\beta}\}^{-\ff{d+2}2}\ff 1 {t\E^\nu[1_{\{t<\si_\tau^B\}}]}\int_0^t\|\xi_s\|_{L^1(\mu_0)}\,\d s\\
&\le c_2 \|h\phi_0^{-1}\|_{L^p(\mu_0)} t^{\ff{(d+2)\beta}2}\e^{-[B(\ll_1)-B(\ll_0)]t},\quad t\geq t_0.
\end{split}\end{equation}
According to \eqref{EIG0} and the definition of $A_t$ in \eqref{RAI}, we have
$$P_{t^{-\beta}}^0A_t=\ff 1 {t\E^\nu[1_{\{t<\si_\tau^B\}}]}\sum_{m=1}^\infty\ff{[\mu(\phi_0)\nu(\phi_m)+\nu(\phi_0)\mu(\phi_m)]\e^{-B(\ll_m)t}}
{[B(\ll_m)-B(\ll_0)]\e^{(\ll_m-\ll_0)t^{-\beta}}}\phi_m\phi_0^{-1},\quad t>0.$$
Combining this with \eqref{SDHT-U}, Lemma \ref{Le4.1} and $B(\ll_m)\ge B(\ll_0)$ for each $m\in\mathbb{N}$, we easily deduce that there exist constants $c_3,t_0>0$ such that
\begin{equation}\begin{split}\label{1Le4.6+}
\|P_{t^{-\beta}}^0A_t\|_\infty\le \|\tilde{\rho}_{t,\beta}^{B,\nu}\|_\infty\le c_3 t^{\ff{(2d-2\aa+1)\beta}2-1},\quad t\ge t_0.
\end{split}\end{equation}
By Lemma \ref{density-h}, \eqref{1Le4.6} and \eqref{1Le4.6+}, we obtain
$$\|\rho_{t,\beta}^{B,\nu}\|_\infty\le c_2 \|h\phi_0^{-1}\|_{L^p(\mu_0)}  t^{\ff{(d+2)\beta}2}\e^{-[B(\ll_1)-B(\ll_0)]t}+c_3 t^{\ff{(2d-2\aa+1)\beta}2-1},\quad t\ge t_0.
$$

Since $\beta\in(0,\ff{2}{2d-2\aa+1})$, by Lemma \ref{Le4.1},
we can find a constant $t_1>0$ such that for any $t\ge t_1$, $\|\rho_{t,\beta}^{B,\nu}\|_\infty\le\ff 1 2$. So, $1+\rho_{t,\beta}^{B,\nu}\ge\ff 1 2$, $t\ge t_1$. Hence, by \eqref{wang-inequ}, we have
\begin{equation}\begin{split}\label{1+Le4.6}
\W_2(\mu_t^{B,\nu},\mu_{t,\beta}^{B,\nu})^2\le 8\int_M|\nn(-\L_0)^{-1}(\rho_t^{B,\nu}-\rho_{t,\beta}^{B,\nu})|^2\,\d\mu_0,\quad t\ge t_1.
\end{split}\end{equation}

It suffices to estimate the right hand side of \eqref{1+Le4.6}. Let $\epsilon=t^{-\beta}$. By the fact that $(-\mathcal{L}_0)^{-1/2}=a\int_0^\infty P_{s^2}^0\,\d s$ with $a=\frac{2}{\sqrt{\pi}}$, we have
\begin{equation}\begin{split}\label{1L4.1-1}
{\rm J}&:=\|\nabla(-\mathcal{L}_0)^{-1}(\rho_t^{B,\nu}-\rho_{t,\beta}^{B,\nu})\|_{L^2(\mu_0)}
=\|(-\mathcal{L}_0)^{-1/2}(\rho_t^{B,\nu}-\rho_{t,\beta}^{B,\nu})\|_{L^2(\mu_0)}\\
&=a\Big\|\int_0^\infty P_{s^2}^0(\rho_t^{B,\nu}-P_\epsilon^0\rho_t^{B,\nu})\,\d s \Big\|_{L^2(\mu_0)} =\frac{a}{2}\Big\|\int_0^\infty\frac{1}{\sqrt{r}}( P_r^0\rho_t^{B,\nu}
-P_{r+\epsilon}^0\rho_t^{B,\nu})\,\d r \Big\|_{L^2(\mu_0)}\\
&=\frac{a}{2}\Big\|\int_0^\infty\frac{1}{\sqrt{r}} P_r^0\rho_t^{B,\nu}\,\d r
-\int_\epsilon^\infty \frac{1}{\sqrt{r-\epsilon}}P_r^0\rho_t^{B,\nu}\,\d r \Big\|_{L^2(\mu_0)}\\
&=\frac{a}{2}\Big\|\int_\epsilon^\infty \Big(\frac{1}{\sqrt{r-\epsilon}}-\frac{1}{\sqrt{r}}\Big)P_r^0\rho_t^{B,\nu}\,\d r -\int_0^\epsilon\frac{1}{\sqrt{r}} P_r^0\rho_t^{B,\nu}\,\d r\Big\|_{L^2(\mu_0)}\\
&\leq \frac{a}{2}({\rm J}_1+{\rm J_2}),
\end{split}\end{equation}
where we applied twice the change-of-variables method, and let
$${\rm J}_1:=\int_\epsilon^\infty\Big(\frac{1}{\sqrt{r-\epsilon}}-\frac{1}{\sqrt{r}}\Big)\|P_r^0\rho_t^{B,\nu}\|_{L^2(\mu_0)}\,\d r,
\quad {\rm J}_2:=\int_0^\epsilon\frac{1}{\sqrt{r}}\| P_r^0\rho_t^{B,\nu}\|_{L^2(\mu_0)}\,\d r. $$

(2) Next we move to estimate ${\rm J}_1$ and ${\rm J}_2$ respectively.

(i) To estimate ${\rm J}_1$,  we aim to show that, for every $p\in(p_0,\infty]$, there exist constants $c,t_0>0$ such that
\begin{equation}\begin{split}\label{2Lee4.6}
{\rm J}_{1,1}&:=\int_\epsilon^\infty\Big(\frac{1}{\sqrt{r-\epsilon}}-\frac{1}{\sqrt{r}}\Big)(\|P_r^0\tilde{\rho}_t^{B,\nu}\|_{L^2(\mu_0)}+\|P_r^0 A_t\|_{L^2(\mu_0)})\,\d r\\
&\le c\|h\phi_0^{-1}\|_{L^p(\mu_0)}t^{-(1+\beta/2)},\quad t\ge t_0,
\end{split}\end{equation}
and
\begin{equation}\begin{split}\label{3Le4.6}
{\rm J}_{1,2}&:=\int_\epsilon^\infty\Big(\frac{1}{\sqrt{r-\epsilon}}-\frac{1}{\sqrt{r}}\Big)\ff 1{t\E^\nu[1_{\{t<\si_\tau^B\}}]}\left\|\int_0^t P_r^0\xi_s\,\d s\right\|_{L^2(\mu_0)}\,\d r\\
&\le c\e^{-[B(\ll_1)-B(\ll_0)]t}\|h\phi_0^{-1}\|_{L^{p}(\mu_0)}t^{-\beta/2},\quad t\ge t_0.
\end{split}\end{equation}

\underline{\emph{Proof of} \eqref{2Lee4.6}.}  By the expression of $\tilde{\rho}_t^{B,\nu}$, $A_t$ in \eqref{RAI} and \eqref{EIG0},
\begin{align*}&P_r^0\tilde{\rho}_t^{B,\nu}=\frac{\e^{-B(\lambda_0)t}}{t\E^{\nu}[1_{\{t<\sigma_\tau^B\}}]}
\sum_{m=1}^\infty\frac{\mu(\phi_0)\nu(\phi_m)+\nu(\phi_0)\mu(\phi_m)}{B(\lambda_m)-B(\lambda_0)}
\e^{-(\lambda_m-\lambda_0)r}\phi_m\phi_0^{-1},
\end{align*}
and
\begin{align*}
P_r^0A_t=\frac{1}{t\E^{\nu}[1_{\{t<\sigma_\tau^B\}}]}
\sum_{m=1}^\infty\frac{[\mu(\phi_0)\nu(\phi_m)+\nu(\phi_0)\mu(\phi_m)]\e^{-B(\lambda_m)t}}{B(\lambda_m)-B(\lambda_0)}
\e^{-(\lambda_m-\lambda_0)r}\phi_m\phi_0^{-1}.
\end{align*}
Since $\{\phi_m\phi_0^{-1}\}_{m\in\mathbb{N}_0}$ is an eigenbasis of $\mathcal{L}_0$ in $L^2(\mu_0)$, by \eqref{SDHT-U}, 
\eqref{equ-B} and \eqref{B-lb}, we derive that
\begin{equation}\begin{split}\label{1L4.1-2++}
&\|P_r^0 A_t\|_{L^2(\mu_0)}\le\big\|P_r^0\tilde{\rho}_t^{B,\nu}\big\|_{L^2(\mu_0)}\\&=\frac{\e^{-B(\lambda_0)t}}{t\E^{\nu}[1_{\{t<\sigma_\tau^B\}}]}\Big(\sum_{m=1}^\infty
\frac{[\mu(\phi_0)\nu(\phi_m)+\nu(\phi_0)\mu(\phi_m)]^2}{[B(\lambda_m)-B(\lambda_0)]^2}
\e^{-2(\lambda_m-\lambda_0)r}\Big)^{1/2}\\
&\leq \frac{c_4}{t}\Big(\sum_{m=1}^\infty
\frac{[\mu(\phi_0)\nu(\phi_m)+\nu(\phi_0)\mu(\phi_m)]^2}{(\lambda_m-\lambda_0)^{2\alpha}}
\e^{-2(\lambda_m-\lambda_0)r}\Big)^{1/2},\quad t\geq t_0,
\end{split}\end{equation}
for some constant $c_4>0$.

Let $\hat{h}=\mu(\phi_0)h\phi_0^{-1}+\nu(\phi_0)\phi_0^{-1}$. If $p_0<p<2\vee p_0$, which is equivalent to that $p_0\leq2$ and $p_0<p<2$, then
\begin{equation}\begin{split}\label{1L4.1-2+}
\|\hat{h}\|_{L^{p}(\mu_0)}&\leq\mu(\phi_0)\|h\phi_0^{-1}\|_{L^{p}(\mu_0)}+\nu(\phi_0)\|\phi_0^{-1}\|_{L^{p}(\mu_0)}\\
&\leq\mu(\phi_0)\|h\phi_0^{-1}\|_{L^{p}(\mu_0)}+\|h\phi_0^{-1}\|_{L^1(\mu_0)}\|\phi_0^{-1}\|_{L^2(\mu_0)}\\
&\le 2\|h\phi_0^{-1}\|_{L^{p}(\mu_0)},
\end{split}\end{equation}
and if $\infty\geq p\geq2\vee p_0$, then
\begin{equation}\begin{split}\label{1L4.1-2-}
\|\hat{h}\|_{L^2(\mu_0)}\leq\|h\phi_0^{-1}\|_{L^2(\mu_0)}+\|h\phi_0^{-1}\|_{L^1(\mu_0)}\leq 2\|h\phi_0^{-1}\|_{L^p(\mu_0)},
\end{split}\end{equation}
since $\mu(\phi_0)\le 1$ and $\|\phi_0^{-1}\|_{L^2(\mu_0)}=1$.
By \eqref{SG0},
$$(P_r^0-\mu_0)\hat{h}=\sum_{m=1}^\infty [\mu(\phi_0)\nu(\phi_m)+\nu(\phi_0)\mu(\phi_m)]\e^{-(\lambda_m-\lambda_0)r}\phi_m\phi_0^{-1},$$
which immediately leads to
\begin{equation}\begin{split}\label{1L4.1-3}
&\big\|(-\mathcal{L}_0)^{-\alpha}(P_r^0-\mu_0)\hat{h}\big\|_{L^2(\mu_0)}\\
&=\Big\|\sum_{m=1}^\infty\frac{\mu(\phi_0)\nu(\phi_m)+\nu(\phi_0)\mu(\phi_m)}{(\lambda_m-\lambda_0)^{\alpha}}\e^{-(\lambda_m-\lambda_0)r}\phi_m\phi_0^{-1} \Big\|_{L^2(\mu_0)}\\
&=\Big(\sum_{m=1}^\infty\frac{[\mu(\phi_0)\nu(\phi_m)+\nu(\phi_0)\mu(\phi_m)]^2}{(\lambda_m-\lambda_0)^{2\alpha}}
\e^{-2(\lambda_m-\lambda_0)r}\Big)^{1/2}.
\end{split}\end{equation}
Due to \eqref{PI0}, \eqref{1L4.1-2++}, \eqref{1L4.1-3}, \eqref{1L4.1-2+} (resp. \eqref{1L4.1-2-}) and the fact that $(-\mathcal{L}_0)^{-\alpha}=C\int_0^\infty P_{s^{1/\alpha}}^0\,\d s$ for some constant $C>0$, we have
\begin{equation*}\begin{split}\label{2Le4.6}
{\rm J}_{1,1}
&\leq \frac{c}{t}\int_\epsilon^\infty\Big(\frac{1}{\sqrt{r-\epsilon}}-\frac{1}{\sqrt{r}}\Big)
\int_0^\infty \|(P^0_{r+s^{1/\alpha}}-\mu_0)\hat{h}\|_{L^2(\mu_0)}\,\d s \d r\\
&\leq\frac{c}{t}\int_\epsilon^\infty\Big(\frac{1}{\sqrt{r-\epsilon}}-\frac{1}{\sqrt{r}}\Big)
\int_0^\infty \|P^0_{r+s^{1/\alpha}}-\mu_0\|_{L^{p}(\mu_0)\rightarrow L^2(\mu_0)}\|\hat{h}\|_{L^{p}(\mu_0)}\,\d s \d r\\
&\leq\frac{c}{t}\|h\phi_0^{-1}\|_{L^{p}(\mu_0)}\int_\epsilon^\infty\Big(\frac{1}{\sqrt{r-\epsilon}}-\frac{1}{\sqrt{r}}\Big)\int_0^\infty \e^{-(\lambda_1-\lambda_0)(r+s^{1/\alpha})}\big[1\wedge (r+s^{1/\aa})\big]^{-\ff{(d+2)(2-p)}{4p}}\,\d s \d r\\
&\leq\frac{c}{t}\|h\phi_0^{-1}\|_{L^{p}(\mu_0)}\int_\epsilon^\infty\Big(\frac{1}{\sqrt{r-\epsilon}}-\frac{1}{\sqrt{r}}\Big)\,\d r
\int_0^\infty\e^{-(\ll_1-\ll_0)s^{1/\aa}}(1\wedge{s^{1/\aa}})^{-\ff{(d+2)(2-p)}{4p}}\,\d s\\
&\leq c\|h\phi_0^{-1}\|_{L^{p}(\mu_0)}t^{-(1+\beta/2)},\quad  t\geq t_0,\\
\end{split}\end{equation*}
where the positive constant $c$ may vary from line to line, and in the last inequality we used the fact that
$$\int_\epsilon^\infty\Big(\ff 1{\sqrt{r-\epsilon}}-\ff 1 {\sqrt{r}}\Big)\,\d r=\ff{2\sqrt{\epsilon}}{\sqrt{\pi}},\quad\int_0^\infty\e^{-(\ll_1-\ll_0)s^{1/\aa}}(1\wedge{s^{1/\aa}})^{-\ff{(d+2)(2-p)}{4p}}\,\d s<\infty,$$
since $p>p_0>2(d+2)/(d+2+4\aa)$.

\underline{\emph{Proof of} \eqref{3Le4.6}.} Suppose that $p_0<p\leq 6$. (Note that $p_0<6$.) By H\"{o}lder's inequality, \eqref{PQ0} and \eqref{PHI}, there exist constants $c_5,c_6>0$ such that, for any $k,l>0$,
\begin{align*}
&\|[\eta_l^\nu-\nu(\phi_0)][P_k^0\phi_0^{-1}-\mu(\phi_0)]\|_{L^2(\mu_0)}\\
&\le\|P_l^0(h\phi_0^{-1})-\mu_0(h\phi_0^{-1})\|_{L^q(\mu_0)}\|P_k^0\phi_0^{-1}-\mu(\phi_0)\|_{L^{\ff{2q}{q-2}}(\mu_0)}\\
&\le\|P_l^0-\mu_0\|_{L^{p}(\mu_0)\to L^q(\mu_0)}\|h\phi_0^{-1}\|_{L^{p}(\mu_0)}\|P_k^0(\phi_0^{-1})-\mu(\phi_0)\|_{L^{\ff{2q}{q-2}}(\mu_0)}\\
&\le c_5\e^{-(\ll_1-\ll_0)(k+l)}(1\wedge l)^{-\ff{(d+2)(q-p)}{2pq}}\|h\phi_0^{-1}\|_{L^{p}(\mu_0)}\|\phi_0^{-1}\|_{L^{\ff{2q}{q-2}}(\mu_0)}\\
&\le c_6\e^{-(\ll_1-\ll_0)(k+l)}(1\wedge l)^{-\ff{(d+2)(q-p)}{2pq}}\|h\phi_0^{-1}\|_{L^{p}(\mu_0)},\quad  q\in (6,\infty].
\end{align*}
By \eqref{NU}, for every $p\in(p_0,\infty]$, we find a constant $c_7>0$ such that
$$|\nu(\phi_0\{\mu(\phi_0)-P_l^0\phi_0^{-1}\})|\le c_7 \e^{-(\ll_1-\ll_0)l}\|h\phi_0^{-1}\|_{L^p(\mu_0)},\quad l>0.$$
Hence, according to \eqref{LT} and \eqref{int-rep-xi}, we have
\begin{align*}
&\|\xi_s\|_{L^2(\mu_0)}\\
&\le\|h\phi_0^{-1}\|_{L^{p}(\mu_0)}\int_0^\infty\int_0^\infty \e^{-\ll_0(k+l)}\e^{-(\ll_1-\ll_0)(k+l)}(1\wedge l)^{-\ff{(d+2)(q-p)}{2pq}}\,\P(S_{t-s}^B\in\d k)\P(S_s^B\in\d l)\\
&\quad+\|h\phi_0^{-1}\|_{L^p(\mu_0)}\int_0^\infty \e^{-\ll_0 l}\e^{-(\ll_1-\ll_0)l}\P(S_t^B\in\d l)\\
&\le c\|h\phi_0^{-1}\|_{L^{p}(\mu_0)}\e^{-B(\ll_1)(t-s)}\int_0^\infty \e^{-\ll_1 l}(1\wedge l)^{-\ff{(d+2)(q-p)}{2pq}}\P(S_s^B\in\d l)\\
&\quad+ c\|h\phi_0^{-1}\|_{L^p(\mu_0)}\e^{-B(\ll_1)t}\\
&\le c\|h\phi_0^{-1}\|_{L^{p}(\mu_0)}\e^{-B(\ll_1)t}\Big(1+s^{-\ff{(d+2)(q-p)}{2\aa pq}}\Big)+ c\|h\phi_0^{-1}\|_{L^p(\mu_0)}\e^{-B(\ll_1)t},
\end{align*}
where the positive constant $c$ may vary from line to line. A similar argument as in \eqref{10+Th3.4+} leads to
\begin{equation}\begin{split}\label{3Le4.6-}
\|P_r^0\xi_s\|_{L^2(\mu_0)}\le c_8\|h\phi_0^{-1}\|_{L^{p}(\mu_0)}\e^{-B(\ll_1)t}\Big(1+s^{-\ff{(d+2)(q-p)}{2\aa pq}}\Big),\quad r>0,
\end{split}\end{equation}
for some constant $c_8>0$.

It is easy to see that
$$0\le \ff{(d+2)(6-p)}{12\aa p}<1,\quad p\in(p_0,6].$$
Then, for every $p\in (p_0,6]$, there exists $\bar{p}\in(6,\infty]$ such that
$$\vartheta:=\ff{(d+2){(\bar{p}-p)}}{2\aa p\bar{p}}\in(0,1).$$
Hence, by \eqref{3Le4.6-} and \eqref{SDHT-U}, we can find constants $c_{9},c_{10},t_0>0$ such that
\begin{equation*}\begin{split}
\ff{1}{t\E^\nu[1_{\{t<\si_\tau^B\}}]}\left\|\int_0^t P_r^0 \xi_s\,\d s\right\|_{L^2(\mu_0)}&\le\ff{1}{t\E^\nu[1_{\{t<\si_\tau^B\}}]}\int_0^t\|P_r^0\xi_s\|_{L^2(\mu_0)}\,\d s\\
&\le\ff{c_{9}\e^{-B(\ll_1)t}}{t\E^\nu[1_{\{t<\si_\tau^B\}}]}\|h\phi_0^{-1}\|_{L^{p}(\mu_0)}\int_0^t(1+s^{-\vartheta})\,\d s\\
&\le c_{10}\e^{-[B(\ll_1)-B(\ll_0)]t}\|h\phi_0^{-1}\|_{L^{p}(\mu_0)},\quad t\geq t_0.
\end{split}\end{equation*}
Thus, for every $p\in(p_0,6]$, there exist constants $t_0,c_{11}>0$ such that
\begin{equation*}\begin{split}\label{3Le4.6-A}
{\rm J}_{1,2}
&\le c_{10}\e^{-[B(\ll_1)-B(\ll_0)]t}\|h\phi_0^{-1}\|_{L^{p}(\mu_0)}\int_\epsilon^\infty\Big(\frac{1}{\sqrt{r-\epsilon}}-\frac{1}{\sqrt{r}}\Big)\,\d r\\
&\le c_{11}\e^{-[B(\ll_1)-B(\ll_0)]t}\|h\phi_0^{-1}\|_{L^{p}(\mu_0)}t^{-\beta/2},\quad t\ge t_0.
\end{split}\end{equation*}

Suppose that $p\in(6,\infty]$. Then there exists constant $c_{12}>0$ such that, for any $k,l>0$,
\begin{align*}
&\|[\eta_l^\nu-\nu(\phi_0)][P_k^0\phi_0^{-1}-\mu(\phi_0)]\|_{L^2(\mu_0)}\\
&\le\|P_l^0(h\phi_0^{-1})-\mu_0(h\phi_0^{-1})\|_{L^p(\mu_0)}\|P_k^0\phi_0^{-1}-\mu(\phi_0)\|_{L^{\ff{2p}{p-2}}(\mu_0)}\\
&\le c_{12}\e^{-(\ll_1-\ll_0)(k+l)}\|h\phi_0^{-1}\|_{L^p(\mu_0)},
\end{align*}
and a similar argument also leads to  \eqref{3Le4.6}.

Thus, by the definition of $\rho_t^{B,\nu}$, \eqref{2Lee4.6} and \eqref{3Le4.6} imply that, for every $p\in (p_0,\infty]$, there exist constants $c_{13},t_0>0$ such that
\begin{equation}\label{4Le4.6}
{\rm J_1}\le  c_{13}\|h\phi_0^{-1}\|_{L^p(\mu_0)}t^{-(1+\beta/2)},\quad t\ge t_0.
\end{equation}

(ii) To estimate ${\rm J}_2$,  by an analogous argument for \eqref{2Lee4.6} and \eqref{3Le4.6}, we have
\begin{equation*}\begin{split}\label{1L4.1-6}
&\int_0^\epsilon\frac{1}{\sqrt{r}}(\|P_r^0\tilde{\rho}_t^{B,\nu}\|_{L^2(\mu_0)}+\|P_r^0 A_t^{B,\nu}\|_{L^2(\mu_0)})\,\d r\\
&\leq \frac{c}{t}\int_0^\epsilon \frac{1}{\sqrt{r}}\|(-\mathcal{L}_0)^{-\alpha}(P_r^0-\mu_0)\hat{h}\|_{L^2(\mu_0)}\,\d r\\
&\le\frac{c}{t}\|h\phi_0^{-1}\|_{L^p(\mu_0)}\int_0^\epsilon\ff {\d r}{\sqrt{r}}\int_0^\infty
\e^{-(\ll_1-\ll_0)s^{1/\aa}}(1\wedge s^{1/\aa})^{-\ff{(d+2)(2-p)}{4p}}\,\d s\\
&\le c\|h\phi_0^{-1}\|_{L^p(\mu_0)} t^{-(1+\beta/2)},\quad t\ge t_0.
\end{split}\end{equation*}
and
\begin{equation}\begin{split}\label{7Le4.6}
&\int_0^\epsilon\frac{1}{\sqrt{r}}\ff 1{t\E^\nu[1_{\{t<\si_\tau^B\}}]}\left\|\int_0^t P_r^0\xi_s\d s\right\|_{L^2(\mu_0)}\,\d r\\
&\le c \|h\phi_0^{-1}\|_{L^p(\mu_0)}\e^{-[B(\ll_1)-B(\ll_0)]t}\int_0^\epsilon\frac{1}{\sqrt{r}}\,\d r\\
&\le c \|h\phi_0^{-1}\|_{L^p(\mu_0)}\e^{-[B(\ll_1)-B(\ll_0)]t}t^{-\beta/2},\quad t\ge t_0,
\end{split}\end{equation}
where the constant $c>0$ may vary from line to line. Thus, there exists a constant $c_{14}>0$, such that
\begin{equation}\label{5Le4.6}
{\rm J_2}\le c_{14} \|h\phi_0^{-1}\|_{L^p(\mu_0)}t^{-(1+\beta/2)},\quad t\ge t_0.
\end{equation}

(3) Therefore, by \eqref{1+Le4.6}, \eqref{1L4.1-1}, \eqref{4Le4.6}, \eqref{5Le4.6},  we complete the proof of Lemma \ref{Le4.6}.
\end{proof}

An alternative proof leads to the following result, which improves the rate of convergence in the case when $\aa\in(1/2,1]$.
However, $\alpha=1/2$ seems critical for the approach employed below. We postpone the proof of Remark \ref{impr-rate} to the end of this subsection.
\begin{rem}\label{impr-rate}
Assume that $\aa\in(1/2,1]$ and $B\in\textbf{B}^\aa$. Let $\beta\in(0,\ff 2{2d-2\aa+1})$ and $p\in(p_0,\infty]$. Then there exist constants $c,t_0>0$ such that, for any $t\ge t_0$ and any $\nu=h\mu\in\scr{P}_0$ with $h\phi_0^{-1}\in L^p(\mu_0)$,
$$
t^2\W_2(\mu_t^{B,\nu},\mu_{t,\beta}^{B,\nu})^2\le c  \|h\phi_0^{-1}\|_{L^p(\mu_0)}^2  t^{- 2\beta}.
$$
\end{rem}

The next proposition establishes the upper bound in Theorem \ref{T1.2}.
\begin{prp}\label{Th4.7}
Let $\aa\in(0,1]$, $B\in\textbf{B}^\aa$ and $p\in(p_0,\infty]$. Then for any $\nu=h\mu\in\scr{P}_0$ with $h\phi_0^{-1}\in L^p(\mu_0)$,
$$
\limsup_{t\to\infty}\{t^2\W_2(\mu_{t}^{B,\nu},\mu_0)^2\}\le
\ff{1}{\{\mu(\phi_0)\nu(\phi_0)\}^2}\sum_{m=1}^\infty\ff{[\mu(\phi_0)\nu(\phi_m)+\nu(\phi_0)\mu(\phi_m)]^2}{(\ll_m-\ll_0)[B(\ll_m)-B(\ll_0)]^2}.
$$
\end{prp}
\begin{proof}
By the triangle inequality of $\W_2$, we see that for any $\beta\in(0,\ff 2{2d-2\aa+1})$, there exists some constant $t_0>0$ such that
$\tilde{\mu}_{t,\beta}^{B,\nu}\in\scr{P}_0$ for every $t\geq t_0$ and
\begin{equation*}\begin{split}
t^2\W_2(\mu_t^{B,\nu},\mu_0)^2&\le(1+\dd)t^2\W_2(\tilde{\mu}_{t,\beta}^{B,\nu},\mu_0)^2+2(1+\dd^{-1})t^2\W_2(\mu_{t,\beta}^{B,\nu},\tilde{\mu}_{t,\beta}^{B,\nu})^2
\\
&\quad+2(1+\dd^{-1})t^2\W_2(\mu_t^{B,\nu},\mu_{t,\beta}^{B,\nu})^2,\quad t\geq t_0,\,\delta>0.
\end{split}\end{equation*}
According to this and Lemmas \ref{Le4.2}, \ref{Lemma4.4} and \ref{Le4.6}, for any $\beta\in(0,\ff 2 {2d-2\aa+1})$ and any $p\in(p_0,\infty]$, there exist some constants $t_0,c>0$ such that, for every $t\geq t_0$,
\begin{equation*}\begin{split}
t^2\W_2(\mu_t^{B,\nu},\mu_0)^2
&\le(1+\dd)\ff{1+ct^{\ff{(2d-2\aa+1)\beta}{2}-1}}{\{\mu(\phi_0)\nu(\phi_0)\}^2}\sum_{m=1}^\infty\ff{[\mu(\phi_0)\nu(\phi_m)+\nu(\phi_0)\mu(\phi_m)]^2}{(\ll_m-\ll_0)
[B(\ll_m)-B(\ll_0)]^2}\e^{-2(\ll_m-\ll_0)t^{-\beta}}\\
&\quad+2c(1+\dd^{-1})t^2 \|h\phi_0^{-1}\|_{L^p(\mu_0)}\e^{-[B(\ll_1)-B(\ll_0)]t}+2c(1+\dd^{-1})\|h\phi_0^{-1}\|_{L^p(\mu_0)}^2 t^{-\beta}.
\end{split}\end{equation*}
Since $\|h\phi_0^{-1}\|_{L^p(\mu_0)}<\infty$, by letting $t\to\infty$ first and then $\dd\to 0$, we have
\begin{equation*}\begin{split}
\lim_{t\rightarrow\infty}\{t^2\W_2(\mu_t^{B,\nu},\mu_0)^2\}\le\ff 1 {\{\mu(\phi_0)\nu(\phi_0)\}^2}\sum_{m=1}^\infty\ff{[\mu(\phi_0)\nu(\phi_m)+\nu(\phi_0)\mu(\phi_m)]^2}{(\ll_m-\ll_0)
[B(\ll_m)-B(\ll_0)]^2}.
\end{split}\end{equation*}
We finish the proof.
\end{proof}

To end this subsection, we present the proof of Remark \ref{impr-rate}.
\begin{proof}
We use the same notations as in Lemma \ref{Le4.6}. The positive constant $c$ used below may vary from line to line. By the definition of $\rho_t^{B,\nu},\rho_{t,\beta}^{B,\nu}$ and \eqref{RAI}, we have
\begin{equation}\begin{split}\label{1Re4.7}
{\rm J}=\|\nn(-\L_0)^{-1}(\rho_t^{B,\nu}-\rho_{t,\beta}^{B,\nu})\|_{L^2(\mu_0)}\le\tilde{{\rm J}}_1+\tilde{{\rm J}}_2+\tilde{{\rm J}}_3,
\end{split}\end{equation}
where
\begin{equation*}\begin{split}
&\tilde{{\rm J}}_1:=\|\nn(-\L_0)^{-1}(\tilde{\rho}_t^{B,\nu}-\tilde{\rho}_{t,\beta}^{B,\nu})\|_{L^2(\mu_0)},\\
&\tilde{{\rm J}}_2:=\|\nn(-\L_0)^{-1}(A_t-P_{t^{-\beta}}^0 A_t)\|_{L^2(\mu_0)},\\
&\tilde{{\rm J}}_3:=\ff{1}{t\E^\nu[1_{\{t<\si_\tau^B\}}]}\int_0^t\|\nn(-\L_0)^{-1}(\xi_s-P_{t^{-\beta}}^0(\xi_s))\|_{L^2(\mu_0)}\,\d s.
\end{split}\end{equation*}

Noting that
$$\tilde{\rho}_t^{B,\nu}-\tilde{\rho}_{t,\beta}^{B,\nu}=\tilde{\rho}_t^{B,\nu}-P_{t^{-\beta}}^0\tilde{\rho}_{t}^{B,\nu}
=\int_0^{t^{-\beta}}(-\mathcal{L}_0)P_r^0\tilde{\rho}_t^{B,\nu}\,\d r,$$
we have
\begin{equation}\begin{split}\label{1L4.1-1'}
\tilde{{\rm J}}_1&=\|\nabla(-\mathcal{L}_0)^{-1}(\tilde{\rho}_t^{B,\nu}-\tilde{\rho}_{t,\beta}^{B,\nu})\|_{L^2(\mu_0)}
=\|(-\mathcal{L}_0)^{-1/2}(\tilde{\rho}_t^{B,\nu}-\tilde{\rho}_{t,\beta}^{B,\nu})\|_{L^2(\mu_0)}\\
&=\Big\|\int_0^\epsilon(-\mathcal{L}_0)^{1/2}P_r^0\tilde{\rho}_t^{B,\nu}\,\d r \Big\|_{L^2(\mu_0)}
\leq\int_0^\epsilon\big\|(-\mathcal{L}_0)^{1/2}P_r^0\tilde{\rho}_t^{B,\nu}\big\|_{L^2(\mu_0)}\,\d r.
\end{split}\end{equation}
By the expression of $\tilde{\rho}_t^{B,\nu}$ in \eqref{RAI} and \eqref{EIG0},
$$(-\mathcal{L}_0)^{1/2}P_r^0\tilde{\rho}_t^{B,\nu}=a_t
\sum_{m=1}^\infty\frac{ \mu(\phi_0)\nu(\phi_m)+\nu(\phi_0)\mu(\phi_m) }{B(\lambda_m)-B(\lambda_0)}\sqrt{\lambda_m-\lambda_0}\,
\e^{-(\lambda_m-\lambda_0)r}\phi_m\phi_0^{-1},$$
where $a_t:=\frac{\e^{-B(\lambda_0)t}}{t\E^{\nu}[1_{\{t<\sigma_\tau^B\}}]}$.
Since $(\phi_m\phi_0^{-1})_{m\in\mathbb{N}_0}$ is an eigenbasis of $\mathcal{L}_0$ in $L^2(\mu_0)$, by \eqref{SDHT-U},
\eqref{equ-B} and \eqref{B-lb}, we derive that
\begin{equation}\begin{split}\label{1L4.1-2}
&\big\|(-\mathcal{L}_0)^{1/2}P_r^0\tilde{\rho}_t^{B,\nu}\big\|_{L^2(\mu_0)}\\
&=a_t\Big(\sum_{m=1}^\infty
\frac{[\mu(\phi_0)\nu(\phi_m)+\nu(\phi_0)\mu(\phi_m)]^2}{[B(\lambda_m)-B(\lambda_0)]^2}(\lambda_m-\lambda_0)
\e^{-2(\lambda_m-\lambda_0)r}\|\phi_m\phi_0^{-1}\|_{L^2(\mu_0)}^2\Big)^{1/2}\\
&\leq \frac{c}{t}\Big(\sum_{m=1}^\infty
\frac{[\mu(\phi_0)\nu(\phi_m)+\nu(\phi_0)\mu(\phi_m)]^2}{(\lambda_m-\lambda_0)^{2\alpha-1}}
\e^{-2(\lambda_m-\lambda_0)r}\Big)^{1/2},\quad t\geq t_0,
\end{split}\end{equation}
for some constant $c>0$. By a similar argument as $\tilde{{\rm J}}_1$ and \eqref{1L4.1-2}, we have
\begin{equation}\begin{split}\label{1L4.1-2''}
\tilde{{\rm J}}_2\le\int_0^\vv\|(-\L_0)^{\ff 1 2}P_r^0 A_t\|_{L^2(\mu_0)}\,\d r,
\end{split}\end{equation}
and
\begin{equation}\begin{split}\label{1L4.1-2'}
&\big\|(-\mathcal{L}_0)^{1/2}P_r^0 A_t\big\|_{L^2(\mu_0)}\\
&= \ff 1 {t\E^\nu[1_{\{t<\si_\tau^B\}}]}\Big(\sum_{m=1}^\infty
\frac{[\mu(\phi_0)\nu(\phi_m)+\nu(\phi_0)\mu(\phi_m)]^2\e^{-2B(\ll_m)t}}{[B(\lambda_m)-B(\lambda_0)]^2\e^{2(\lambda_m-\lambda_0)r}}(\lambda_m-\lambda_0)
\|\phi_m\phi_0^{-1}\|_{L^2(\mu_0)}^2\Big)^{1/2}\\
&\leq \frac{c}{t}\Big(\sum_{m=1}^\infty
\frac{[\mu(\phi_0)\nu(\phi_m)+\nu(\phi_0)\mu(\phi_m)]^2}{(\lambda_m-\lambda_0)^{2\alpha-1}}
\e^{-2(\lambda_m-\lambda_0)r}\Big)^{1/2},\quad t\geq t_0,
\end{split}\end{equation}
for some constant $c>0$.

Let $\hat{h}=\mu(\phi_0)h\phi_0^{-1}+\nu(\phi_0)\phi_0^{-1}$. Since
$$(P_r^0-\mu_0)\hat{h}=\sum_{m=1}^\infty [\mu(\phi_0)\nu(\phi_m)+\nu(\phi_0)\mu(\phi_m)]\e^{-(\lambda_m-\lambda_0)r}\phi_m\phi_0^{-1},$$
 by \eqref{EIG0}, we immediately have
\begin{equation}\begin{split}\label{1L4.1-3'}
\big\|(-\mathcal{L}_0)^{1/2-\alpha}(P_r^0-\mu_0)\hat{h}\big\|_{L^2(\mu_0)}
=\Big(\sum_{m=1}^\infty\frac{[\mu(\phi_0)\nu(\phi_m)+\nu(\phi_0)\mu(\phi_m)]^2}{(\lambda_m-\lambda_0)^{2\alpha-1}}
\e^{-2(\lambda_m-\lambda_0)r}\Big)^{1/2}.
\end{split}\end{equation}

Thus, combining \eqref{1L4.1-1'}, \eqref{1L4.1-2} \eqref{1L4.1-2''}, \eqref{1L4.1-2'} and \eqref{1L4.1-3'}, we obtain
\begin{equation*}\begin{split}\label{1L4.1-4'}
\tilde{{\rm J}}_1+\tilde{{\rm J}}_2\leq \frac{c}{t}\int_0^\epsilon\big\|(-\mathcal{L}_0)^{1/2-\alpha}(P_r^0-\mu_0)\hat{h}\big\|_{L^2(\mu_0)}\,\d r,\quad t\geq t_0.
\end{split}\end{equation*}

Suppose that $ p_0<p<2\vee p_0$. By \eqref{PI0}, \eqref{1L4.1-2+}  and the fact that $(-\mathcal{L}_0)^{1/2-\alpha}=c\int_0^\infty P_{s^{2/(2\alpha-1)}}^0\,\d s$, we have
\begin{equation}\begin{split}\label{2Re4.7}
\tilde{{\rm J}}_1+\tilde{{\rm J}}_2&\leq \frac{c}{t}\int_0^\epsilon\Big\|\int_0^\infty(P^0_{r+s^{2/(2\alpha-1)}}-\mu_0)\hat{h}\,\d s\Big\|_{L^2(\mu_0)}\,\d r\\
&\leq\frac{c}{t}\int_0^\epsilon\int_0^\infty\big\|(P^0_{r+s^{2/(2\alpha-1)}}-\mu_0)\hat{h}\big\|_{L^2(\mu_0)}\,\d s\d r\\
&\le \frac{c}{t}\int_0^\epsilon\int_0^\infty\|P_{r+s^{2/(2\alpha-1)}}^0-\mu_0\|_{L^p(\mu_0)\to L^2(\mu_0)}\|\hat{h}\|_{L^p(\mu_0)}\,\d s\d r\\
&\le\ff c {t}\|h\phi_0^{-1}\|_{L^p(\mu_0)}\int_0^\epsilon \e^{-(\ll_1-\ll_0)r}\,\d r\int_0^\infty \e^{-(\ll_1-\ll_0)s^{2/(2\aa-1)}}(1\wedge s^{2/(2\aa-1)})^{-\ff{(d+2)(2-p)}{4p}}\,\d s\\
&\le c\|h\phi_0^{-1}\|_{L^p(\mu_0)} t^{-(1+ \beta)},\quad t\geq t_0,
\end{split}\end{equation}
for some constant $c>0$, where we applied the fact that
$$\int_0^\infty \e^{-(\ll_1-\ll_0)s^{2/(2\aa-1)}}(1\wedge s^{2/(2\aa-1)})^{-\ff{(d+2)(2-p)}{4p}}\d s<\infty,$$
since
$$0\leq\ff 2 {2\aa-1}\ff{(d+2)(2-p)}{4p}<1,\quad p\in(p_0,2],\,\alpha\in(1/2,1].$$
Suppose that $ p_0\vee2\leq p\leq\infty$. By an analogous argument, we also have \eqref{2Re4.7}.

Now we turn to estimate $\tilde{{\rm J}}_3$. By \eqref{1L4.1-1}, \eqref{3Le4.6} and \eqref{7Le4.6},
\begin{equation}\begin{split}\label{4Re4.7}
&\tilde{{\rm J}}_3
\le\ff {c}{t\E^\nu[1_{\{t<\si_\tau^B\}}]}\int_0^t\left\|\int_\epsilon^\infty\left(\ff 1{\sqrt{r-\epsilon}}-\ff 1{\sqrt{r}}\right)P_r^0 \xi_s\,\d r-\int_0^\epsilon\ff 1 {\sqrt{r}}P_r^0 \xi_s\,\d r\right\|_{L^2(\mu_0)}\,\d s\\
&\le\ff {c}{t\E^\nu[1_{\{t<\si_\tau^B\}}]}\int_0^t\int_\epsilon^\infty\left(\ff 1{\sqrt{r-\epsilon}}-\ff 1{\sqrt{r}}\right)\|P_r^0 \xi_s\|_{L^2(\mu_0)}\,\d r\d s\\
&\quad+\ff {c}{t\E^\nu[1_{\{t<\si_\tau^B\}}]}\int_0^t\int_0^\epsilon\ff 1{\sqrt{r}}\|P_r^0 \xi_s\|_{L^2(\mu_0)}\,\d r\d s\\
&\le c \|h\phi_0^{-1}\|_{L^p(\mu_0)}\e^{-[B(\ll_1)-B(\ll_0)]t}\left[\int_\epsilon^\infty\left(\ff 1{\sqrt{r-\epsilon}}-\ff 1 {\sqrt{r}}\right)\d r+\int_0^\epsilon\ff 1{\sqrt{r}}\,\d r\right]\\
&\le c \|h\phi_0^{-1}\|_{L^p(\mu_0)}\e^{-[B(\ll_1)-B(\ll_0)]t}t^{-\beta/2},\quad t\geq t_0,
\end{split}\end{equation}
for some constant $c>0$.

Therefore, combining \eqref{1+Le4.6}, \eqref{1Re4.7}, \eqref{2Re4.7} and \eqref{4Re4.7}, we complete the proof.
\end{proof}

\subsection{Lower bounds}
In this subsection, we present the proof of the lower bound in Theorem \ref{T1.2}. The idea is motivated by \cite[Section 4]{eW1}.
Note that, by Lemma \ref{Le4.1} above, for any $\beta\in(0,\ff 2{2d-2\aa+1})$, there exists some constant $t_0>0$ such that
$\tilde{\mu}_{t,\beta}^{B,\nu}\in\scr{P}_0$ for every $t\ge t_0$ and every $\nu\in\scr{P}_0$.

Set $f_{t,\beta}^B:=(-\mathcal{L}_0)^{-1}\tilde{\rho}_{t,\beta}^{B,\nu}$. The next lemma establishes useful regularity estimates for $f_{t,\beta}^B$ and its gradient.
\begin{lem}\label{L-4.1}
For any $\beta>0$, there exists a constant $c>0$ such that
$$\|f_{t,\beta}^B\|_\infty+\|\mathcal{L}_0 f_{t,\beta}^B\|_\infty+\|\nn f_{t,\beta}^B\|_\infty\le c t^{\ff{(5d+2)\bb}4-1},\quad t\ge 1,\,\nu\in\scr{P}_0.$$
\end{lem}
\begin{proof}
By \eqref{EIG0} and \eqref{RAI},
we have, for every $t\geq1$,
\begin{align*}
&(-\mathcal{L}_0) f_{t,\beta}^B=P_{t^{-\beta}}^0\tilde{\rho}_{t}^{B,\nu}\\
&=\ff 1{t\E^\nu[1_{\{t<\si_\tau^B\}}]} \sum_{m=1}^\infty\ff{[\mu(\phi_0)\nu(\phi_m)+\nu(\phi_0)\mu(\phi_m)]\e^{-B(\ll_0)t}}{B(\ll_m)-B(\ll_0)}
\e^{-(\ll_m-\ll_0)t^{-\beta}}\phi_m\phi_0^{-1},
\end{align*}
and
\begin{align*}
f_{t,\beta}^B=\ff 1{t\E^\nu[1_{\{t<\si_\tau^B\}}]} \sum_{m=1}^\infty\ff{[\mu(\phi_0)\nu(\phi_m)+\nu(\phi_0)\mu(\phi_m)]\e^{-B(\ll_0)t}}{(\ll_m-\ll_0)[B(\ll_m)-B(\ll_0)]}
\e^{-(\ll_m-\ll_0)t^{-\beta}}\phi_m\phi_0^{-1}.
\end{align*}
Combining these identities with \eqref{EIG}, \eqref{EIG0UB}, \eqref{SDHT-U}, \eqref{equ-B}, \eqref{B-lb} and the fact that
$$|\mu(\phi_0)\nu(\phi_m)+\nu(\phi_0)\mu(\phi_m)|\le \|\phi_0\|_\infty^2+\|\phi_m\|_\infty^2\le c_1 m,\quad m\in\mathbb{N},$$
for some constant $c_1>0$, we find constants $c_2,c_3,c_4,c_5>0$ such that
\begin{align*}
t\{\|f_{t,\beta}^B\|_\infty+\|\mathcal{L}_0 f_{t,\beta}^B\|_\infty\}&\le c_2\sum_{m=1}^\infty\ff{\e^{-(\ll_m-\ll_0)t^{-\beta}}}{B(\ll_m)-B(\ll_0)}m^{\ff{3d+2}{2d}}\\
&\le c_3\sum_{m=1}^\infty \e^{-\alpha_0^{-1}m^{\ff 2 d}t^{-\beta}}m^{\ff{3d+2}{2d}-\ff{2\aa}{d}}\\
&\le c_4 \int_0^\infty \e^{-\alpha_0^{-1}s^{\ff 2 d}t^{-\beta}}s^{\ff{3d+2-4\aa}{2d}}\,\d s\\
&\le c_5 t^{\ff{5d\beta+(2-4\aa)\beta}{4}},\quad t\geq1,
\end{align*}
where $\alpha_0$ is from \eqref{EIG}. By a similar argument as above, applying \eqref{grad-EIG0UB}, we deduce that there 
exist constants $c_6,c_7,c_8>0$ such that
\begin{align*}
t\|\nn f_{t,\beta}^B\|_\infty&\le c_6\sum_{m=1}^\infty \ff{\e^{-(\ll_m-\ll_0)t^{-\beta}}}{[B(\ll_m)-B(\ll_0)](\ll_m-\ll_0)}m^{\ff{3d+4}{2d}}\\
&\le c_7\sum_{m=1}^\infty \e^{-\alpha_0^{-1}m^{\ff 2 d}t^{-\beta}}m^{\ff{3d+4}{2d}-\ff{2(1+\aa)}{d}}\\
&\le c_8 t^{\ff{\beta(5d-4\aa)}4},\quad t\ge 1.
\end{align*}
The proof is completed.
\end{proof}

In the following lemma, we present the lower bound estimate for $\W_2(\tilde{\mu}_{t,\beta}^{B,\nu},\mu_0)$. With Lemma \ref{L-4.1} in hand, the  proof can be achieved by the same approach employed for \cite[Lemma 4.2]{eW1}. So we omit the details here.
\begin{lem}\label{L4.2}
For any $\beta\in(0,\ff 1 {4(5d+2)}]$, there exist constants $c,t_0>0$ such that
$$t^2\W_2(\tilde{\mu}_{t,\beta}^{B,\nu},\mu_0)^2\ge \hat{I}-c t^{-\ff 1 4},\quad t\geq t_0,\,\nu\in\scr{P}_0,$$
where $$\hat{I}:=\ff 1 {\{\mu(\phi_0)\nu(\phi_0)\}^2}\sum_{m=1}^\infty
\ff{[\mu(\phi_0)\nu(\phi_m)+\nu(\phi_0)\mu(\phi_m)]^2}{(\ll_m-\ll_0)[B(\ll_m)-B(\ll_0)]^2}\e^{-2(\ll_m-\ll_0)t^{-\beta}},\quad t>0,\,\nu\in\scr{P}_0.$$
\end{lem}

The main result of this subsection, which is just the repeat of the lower bound in Theorem \ref{T1.2}, is presented in the next proposition.
\begin{prp}\label{P4.4}
Let $\aa\in(0,1]$, $B\in\textbf{B}^\aa$ and $p\in(p_0,\infty]$. Then for any $\nu=h\mu\in\scr{P}_0$ with $h\phi_0^{-1}\in L^p(\mu_0)$,
\begin{equation*}
\liminf_{t\to\infty}\{t^2\W_2(\mu_t^{B,\nu},\mu_0)^2\}\ge  \ff 1 {\{\mu(\phi_0)\nu(\phi_0)\}^2}\sum_{m=1}^\infty
\ff{[\mu(\phi_0)\nu(\phi_m)+\nu(\phi_0)\mu(\phi_m)]^2}{(\ll_m-\ll_0)[B(\ll_m)-B(\ll_0)]^2}.
\end{equation*}
\end{prp}
\begin{proof}
By the triangular inequality of $\W_2$, we have for any $\beta\in(0,\ff 2{2d-2\aa+1})$,
there exists a constant $t_0>0$ such that  $\tilde{\mu}_{t,\beta}^{B,\nu}\in\scr{P}_0$ for every $t\geq t_0$ and
\begin{equation}\label{2P4.4-}
\W_2(\mu_t^{B,\nu},\mu_0)\ge \W_2(\tilde{\mu}_{t,\beta}^{B,\nu},\mu_0)-
\W_2(\tilde{\mu}_{t,\beta}^{B,\nu},\mu_{t,\beta}^{B,\nu})-\W_2(\mu_{t,\beta}^{B,\nu},\mu_t^{B,\nu}),\quad t\geq t_0.
\end{equation}
Combining Lemmas \ref{Lemma4.4}, \ref{L4.2} and \ref{Le4.6}, for any  $\beta\in(0,\ff 1{4(5d+2)}]$ and any $p\in(p_0,\infty]$, we can find some constants $c,t_0>0$ such that, for every $\nu=h\mu\in\scr{P}_0$ with $h\phi_0^{-1}\in L^p(\mu_0)$  and every $t\ge t_0$,
\begin{align*}
&t\W_2(\tilde{\mu}_{t,\beta}^{B,\nu},\mu_{t,\beta}^{B,\nu})\le ct\e^{-[B(\ll_1)-B(\ll_0)]t/2} \|h\phi_0^{-1}\|_{L^p(\mu_0)}^{\ff 1 2},\\
&t\W_2(\tilde{\mu}_{t,\beta}^{B,\nu},\mu_0)\ge\big( [{\hat{I}}-ct^{-\ff 1 4}]^+\big)^{\ff 1 2},\\
&t\W_2(\mu_{t,\beta}^{B,\nu},\mu_t^{B,\nu})\le c  \|h\phi_0^{-1}\|_{L^p(\mu_0)} t^{-\frac{\beta}{2}}.
\end{align*}
Substituting these estimates into \eqref{2P4.4-}, we immediately obtain that
\begin{equation*}\begin{split}
t\W_2(\mu_t^{B,\nu},\mu_0)&\ge \big([{\hat{I}}-ct^{-\ff 1 4}]^+\big)^{\ff 1 2}-ct\e^{-[B(\ll_1)-B(\ll_0)]t/2}
\|h\phi_0^{-1}\|_{L^p(\mu_0)}^{\ff 1 2} \\
&\quad-c \|h\phi_0^{-1}\|_{L^p(\mu_0)}  t^{-\frac{\beta}{2}},\quad t\ge t_0.
\end{split}\end{equation*}
Since $\|h\phi_0^{-1}\|_{L^p(\mu_0)}<\infty,$ by letting $t\to\infty$, we prove the desired result.
\end{proof}

\subsection*{Acknowledgment}
The work is supported by the National Natural Science Foundation of China (No. 11831014).  The authors would like to thank Prof. Feng-Yu Wang for helpful comments.

\section*{Appendix}
Let  $B\in\mathbf{B}^\alpha$ for some $\alpha\in(0,1]$. Recall that $\mu_0=\phi_0^2\mu$, and $\mu_0$ is called a quasi-ergodic distribution of the $B$-subordinated Dirichlet diffusion process $(X_t^B)_{t\geq0}$ if
for every $\nu\in\scr{P}$ supported on $\mathring{M}$ and every Borel set $E\subset M$,
$$\lim_{t\rightarrow\infty}\E^\nu\Big[\frac{1}{t}\int_0^t 1_E(X_s^B)\,\d s\Big|\sigma_\tau^B>t\Big]=\mu_0(E).$$
The following result  implies that $\mu_0$ is the unique quasi-ergodic distribution of $(X_t^B)_{t\geq0}$.
\renewcommand{\theprp}{\Alph{prp}}
\begin{prp}\label{QED}
Let $\alpha\in(0,1]$ and $B\in\mathbf{B}^\alpha$. Then, for every $\nu\in\scr{P}$ supported on $\mathring{M}$,
\begin{equation*}\label{QED-1}
\lim_{t\rightarrow\infty}\|\mu_t^{B,\nu}-\mu_0\|_{\textup{var}}=0,
\end{equation*}
and
\begin{equation*}\label{QED-2}
\lim_{t\rightarrow\infty}\E^\nu\Big[\frac{1}{t}\int_0^t f(X_s^B)\,\d s\Big|\sigma_\tau^B>t\Big]=\int_M f\,\d\mu_0,\quad  f\in\B_b(M).
\end{equation*}
\end{prp}
\begin{proof} The proof is a direct application of \cite[Theorem 2.1]{CJ2014} (see also \cite{ZLS}). So we only need to check the assumptions $({\rm A}_1)$ and $({\rm A}_2)$ on \cite[page 185]{CJ2014}. $({\rm A}_1)$ is clearly satisfied. By \eqref{SDHK}, Fubini's theorem, $\mu(\phi_m^2)=1$ for each $m\in\mathbb{N}_0$, \eqref{EIG} and \eqref{B-lb},
\begin{equation*}\begin{split}
&\int_M p^{D,B}_t(x,x)\,\mu(\d x)=\int_M \sum_{m=0}^\infty\e^{-B(\lambda_m)t}\phi_m^2(x)\,\mu(\d x)=\sum_{m=0}^\infty\e^{-B(\lambda_m)t}\\
&=\e^{-B(\lambda_0)t}\Big(1+\sum_{m=1}^\infty\e^{-[B(\lambda_m)-B(\lambda_0)]t}\Big)\leq c_1 \Big(1+\sum_{m=1}^\infty\e^{-c_2tm^{2\alpha/d}}\Big)<\infty,\quad t>0,
\end{split}\end{equation*}
for some constants $c_1,c_2>0$. By \eqref{SDSG0}, \eqref{DPQ}, a similar argument as in \eqref{7P3.4-}, there exist constants $c_3,c_4>0$ such that
\begin{equation*}\begin{split}
&\|P_t^{D,B}f\|_{L^\infty(\mu)}=\Big\| \int_0^\infty P_u^Df\,\P(S_t^B\in\d u)\Big\|_{L^\infty(\mu)}\\
&\leq\int_0^\infty \|P_u^Df\|_{L^\infty(\mu)}\,\P(S_t^B\in\d u)\\
& \le c_3 \|f\|_{L^2(\mu)}\E\left[\e^{-\ll_0 S_t^B}(1\wedge S_t^B)^{-\ff d 4}\right]\\
&  \le c_3\|f\|_{L^2(\mu)}\E\left[(1\wedge S_t^B)^{-\ff d 4}\right]\\
& \le c_4\|f\|_{L^2(\mu)}(1+t^{-\ff d{4\aa}}),\quad t>0,\,f\in L^2(\mu),
\end{split}\end{equation*}
i.e., $\|P_t^{D,B}\|_{L^2(\mu)\rightarrow L^\infty(\mu)}\leq c_4(1+t^{-\ff d{4\aa}})$, $t>0$. Thus, $({\rm A}_2)$ is satisfied.
\end{proof}

See e.g. \cite{O2020,HYZ} for recent studies on the quasi-ergodic distribution, and refer to \cite{CMS} for a detailed study on the closely related notions such as the quasi-stationary distribution and the Yaglom limit.

\end{document}